\documentclass[reqno,oneside,12pt,hidelinks]{amsart}
\usepackage{amssymb,amsmath,amsthm,amsxtra,calc,bm, color}
\usepackage{enumerate}
\usepackage[margin=.95in]{geometry}

\usepackage[scr=boondoxo]{mathalfa}
\usepackage{mathtools}

\usepackage{mleftright}
\mleftright

\usepackage{nccmath}
\usepackage{cases}

\usepackage{calc}
\usepackage{graphicx}
\usepackage{hyperref}
\usepackage[final]{microtype}

\renewcommand{\(}{\left(}
\renewcommand{\)}{\right)}
\newcommand{\spmod}[1]{\ensuremath{\,(#1)}}
\newcommand{\Shl}{\operatorname{Sh}}

\renewcommand{\|}{\big |}

\def\Z{\mathbb{Z}}
\def\Q{\mathbb{Q}}
\def\R{\mathbb{R}}
\def\H{\mathbb{H}}
\def\N{\mathbb{N}}
\def\C{\mathbb{C}}

\DeclareMathOperator{\im}{Im}
\DeclareMathOperator{\re}{Re}
\DeclareMathOperator{\Tr}{Tr}
\def\SL{{\rm SL}}
\def\PSL{{\rm PSL}}
\def\GL{{\rm GL}}

\newcommand{\pfrac}[2]{\left(\frac{#1}{#2}\right)}
\newcommand{\pmfrac}[2]{\left(\mfrac{#1}{#2}\right)}
\newcommand{\ptfrac}[2]{\left(\tfrac{#1}{#2}\right)}
\newcommand{\pMatrix}[4]{\left(\begin{matrix}#1 & #2 \\ #3 & #4\end{matrix}\right)}

\renewcommand{\pmatrix}[4]{\left(\begin{smallmatrix}#1 & #2 \\ #3 & #4\end{smallmatrix}\right)}
\renewcommand{\bar}[1]{\overline{#1}}

\newcommand{\Sh}{\mathcal{S}}

\DeclareMathOperator{\He}{He}

\renewcommand{\hat}{\widehat}
\renewcommand{\tilde}{\widetilde}

\renewcommand{\sl}{\big| }
\DeclareMathOperator{\new}{new}
\DeclareMathOperator{\sgn}{sgn}

\def\ep{\varepsilon}

\newtheorem{theorem}{Theorem}[section]
\newtheorem{lemma}[theorem]{Lemma}
\newtheorem{corollary}[theorem]{Corollary}
\newtheorem{proposition}[theorem]{Proposition}

\theoremstyle{remark}
\newtheorem*{remark}{Remark}

\newtheorem*{definition}{Definition}
\numberwithin{equation}{section}

\newcommand{\Frob}{\operatorname{Frob}}
\mathtoolsset{showonlyrefs}

\newcommand{\Gal}{\operatorname{Gal}}

\title[The Shimura lift for modular forms with  the eta multiplier]{The Shimura lift and congruences for modular forms with  the eta multiplier}

\date{\today}
\author{Scott Ahlgren}
\address{Department of Mathematics\\
University of Illinois\\
Urbana, IL 61801} 
\email{sahlgren@illinois.edu} 

 \author{Nickolas Andersen}
 \address{Department of Mathematics, Brigham Young University, Provo, UT 84602}
 \email{nick@math.byu.edu}
\author{Robert Dicks}
\address{Department of Mathematics\\
Clemson University\\
Clemson, SC 29634} 
\email{rdicks@clemson.edu}  
 
\thanks{The first author was  supported by a grant from the Simons Foundation (\#963004 to Scott Ahlgren).
The second author was  supported by a grant from the Simons Foundation (\#854098 to Nickolas Andersen)}
\begin{document}

\begin{abstract}
The Shimura correspondence is a fundamental tool in the study  of half-integral weight modular forms. 
In this paper, we prove a Shimura-type correspondence for  spaces of half-integral weight cusp forms which transform with a power of the Dedekind eta multiplier  twisted by a Dirichlet character. 
We prove that the lift of a cusp form of weight $\lambda+1/2$ and level $N$ has weight $2\lambda$ and level $6N$, and   is new at the primes $2$ and $3$ with specified Atkin-Lehner eigenvalues.
This precise information leads to arithmetic applications.  For a wide family of spaces of half-integral weight modular forms we prove the existence of infinitely many primes $\ell$
which give rise to quadratic congruences modulo arbitrary powers of $\ell$.
\end{abstract}

\maketitle

\setcounter{tocdepth}{1}
\tableofcontents

\section{Introduction and statement of results}

The \emph{Shimura correspondence} \cite{shimura} is a family of maps taking  modular forms of half-integral weight to modular forms of integral weight and preserving the action of the Hecke algebras.  Since its introduction it has been a ubiquitous tool in the study of half-integral weight modular forms.
Works of Waldspurger \cite{Waldspurger} \cite{Waldspurger2} and Kohnen and Zagier \cite{Kohnen-Zagier} establish a connection between the coefficients of half-integral weight forms and the $L$-functions of their Shimura lifts.

Shimura's  construction relies on Weil's converse theorem.
Niwa \cite{niwa} gave a more direct construction of the Shimura lift by integrating a given half-integral weight form against a suitable theta kernel.  This work was refined by Cipra \cite{cipra}, who in particular  extended the results to all positive half-integral weights.
These works concern modular forms 
whose multiplier is a power of $\nu_\theta$ twisted by a Dirichlet character,
where $\nu_\theta$ is the multiplier on $\Gamma_0(4)$ attached to the usual theta function.
If $f$ is such a form on $\Gamma_0(4N)$, then the Shimura lift of $f$ is on $\Gamma_0(2N)$.

Here we will consider modular forms of half-integral weight transforming with a power of the Dedekind eta  multipler $\nu$ twisted by a Dirichlet character (see Section~\ref{sec:background} for precise definitions).
In the simplest case, suppose  that $(r,6)=1$ and that $f$ is a cusp form of weight $\lambda+1/2$ on $\SL_2(\Z)$ with multiplier $\nu^r$.
If $V_m$ denotes the map $z\mapsto mz$, then $f \sl V_{24}$ is a half-integral weight form in the sense of Shimura on $\Gamma_0(576)$ (see Lemma~\ref{lem:eta_to_theta} for details). 
 For a  positive squarefree integer $t$, we can apply the usual Shimura lift $\Shl_t$  to $f(24z)$, which gives a form of weight $2\lambda$ on $\Gamma_0(288)$.
Yang \cite{Yang} showed that in fact we have 
\begin{equation}
\Shl_t(f\sl V_{24}) \in S^{\new}_{2\lambda}\(6, -\ptfrac 8r,-\ptfrac {12}r\) \otimes \ptfrac{12}{\bullet},
\end{equation}
i.e.~there exists a cusp form $g$ of weight $2\lambda$ on $\Gamma_0(6)$ (in the new subspace) with Atkin-Lehner eigenvalues $-\pfrac 8r$ and $-\pfrac {12}r$ at $2$ and $3$, respectively, such that $\Shl_t(f \sl V_{24}) = g \otimes \pfrac{12}{\bullet}$.
Thus $\Shl_t(f \sl V_{24})$ is a cusp form of level $144$
(a similar result holds for the Shimura lift of $f\sl V_8$ when $(r,6)=3$).

Given Yang's result, it is natural to suspect that there exists a modification of the Shimura lift which maps $f$ directly into $S_{2\lambda}^{\new}(6,-\pfrac 8r, -\pfrac{12}r)$.
Here we construct such a lift in a much more general setting.
In particular, if $\psi$ is a Dirichlet character modulo $N$,  we construct a family of lifts which map forms of half-integral weight with multiplier $\psi\nu^r$ on $\Gamma_0(N)$
to forms of integral weight and character $\psi^2$ on  $\Gamma_0(6N)$ and which provide precise information at the primes $2$ and $3$.

The  statement of our results requires some notation. 
Let $\lambda$ and $N$ be positive integers and let $r$ be an odd integer.
Let $\psi$ be a Dirichlet character modulo $N$.
Denote by $S_{\lambda+1/2}(N,\psi\nu^r)$ the space of cusp forms of weight $\lambda+1/2$ on $\Gamma_0(N)$ transforming with multiplier system $\psi\nu^r$
(by \eqref{eq:krcond1} these spaces are trivial unless 
 $\psi(-1) = \pfrac{-1}{r}(-1)^\lambda$). When $\psi$ is trivial we omit it from the notation.
 
In weight $3/2$, we need to avoid unary theta series, so when $(r,6)=1$ we define $S_{3/2}^c(N,\psi\nu^r)$ as the subspace of $S_{3/2}(N,\psi\nu^r)$ comprising forms $f$ which satisfy
\begin{equation}\label{eq:orthogtheta}
    \left\langle f\sl V_{24}, g \right\rangle = 0 \ \text{ for all theta functions }\  g\in S_{\frac 32} \left(576N, \psi \ptfrac{-1}{\bullet}^{\frac{r+1}{2}}\ptfrac{12}\bullet \nu_\theta^{3}\right),
\end{equation}
where $\langle \cdot, \cdot \rangle$ is the usual Petersson inner product.   
We make a similar definition for $S_{3/2}^c(N,\psi\nu^r)$ when $(r,6)=3$ (see Lemma~\ref{lem:eta_to_theta}).

If $(N,6)=1$  then we denote  by $S_{2\lambda}^{\new 2, 3}(6N,\psi^2, \ep_2, \ep_3)$ the space of cusp forms of weight $2\lambda$ on $\Gamma_0(6N)$ with character $\psi^2$ which are new at $2$ and $3$, and with
Atkin-Lehner eigenvalues $\ep_2$ and $\ep_3$ at those primes
(see Section~\ref{sec:intweight} for details).
We make a similar definition for $S_{2\lambda}^{\new 2}(2N,\psi^2, \ep_2)$.
Define 
\begin{equation}\label{eq:epdef}
\ep_{2, r, \psi}:=-\psi(2)\pmfrac{8}{r/(r,3)}, \ \ \ \ 
\ep_{3, r, \psi}:=-\psi(3)\pmfrac{12}{r}.
\end{equation}
Here, and throughout, $\pfrac \bullet\bullet$ denotes the extended quadratic symbol.
For primes $p$, we denote the Hecke operators on $S_{\lambda+1/2}(N,\psi\nu^r)$ and $S_{2\lambda}^{\new 2, 3}(6N,\psi^2, \ep_{2, r, \psi}, \ep_{3, r,\psi})$ by $T_{p^2}$ and $T_p$, respectively (see Section~\ref{sec:background} for details).
Finally, let $L(s,\chi)$ denote the Dirichlet $L$-function.

We can now state the main results, which are slightly different in the cases $(r, 6)=1$ and $(r, 6)=3$.
We note that versions of each theorem can be given without the hypothesis on $N$; see Theorems~\ref{thm:transform} and \ref{thm:transform3}  for details.

\begin{theorem} \label{thm:shimura-lift}
Let $r$ be an integer with $(r,6)=1$ and let  $t$ be  a squarefree positive integer.  Suppose that $\lambda, N \in \Z^{+}$, that $(N, 6)=1$, and that 
 $\psi$ is a Dirichlet character modulo $N$.
Suppose that
\begin{equation}\label{eq:Fdeff}
		F(z) = \sum_{n\equiv r\spmod{24}} a(n) q^\frac n{24} \in  S_{\lambda+\frac12}(N,\psi\nu^r)
	\end{equation}
and if  $\lambda=1$ suppose further that $F\in S_{3/2}^{\rm{c}}(N,\psi\nu^r)$.
Define coefficients $b(n)$ by the relation
\begin{equation}\label{eq:stdef}
		\sum_{n=1}^\infty \frac{b(n)}{n^s} = L\(s-\lambda+1,\psi \ptfrac \bullet t\) \sum_{n=1}^\infty \pfrac{12}{n} \frac{a(tn^2)}{n^s}.
	\end{equation}
Then we have 
	\begin{equation}
		\Sh_t(F) := \sum_{n=1}^\infty b(n)q^n \in S_{2\lambda}^{\new 2, 3}(6N,\psi^2, \ep_{2, r, \psi}, \ep_{3, r,\psi}).
	\end{equation}
Furthermore we have
\begin{equation}\label{equivariance3nmidr}
\Sh_t(T_{p^2}F)=\pmfrac{12}pT_p\,\Sh_t(F)\ \ \ \  \text{for each prime $p\geq 5$}.
\end{equation}
\end{theorem}
\begin{remark}
In this case $\Sh_t(F)=0$ unless $t \equiv r \pmod{24}$.  
\end{remark}
A similar result holds when  $(r,6)=3$; here it is most convenient to write the Fourier expansions in powers of $q^\frac18$.
\begin{theorem} \label{thm:shimura-lift-r=3}
Let $r$ be an integer with $(r,6)=3$ and let  $t$ be  a squarefree positive integer.  Suppose that $\lambda, N \in \Z^{+}$, that $N$ is odd,  and that 
 $\psi$ is a Dirichlet character modulo $N$.
	Suppose that
\begin{equation}\label{eq:Fdef3}
		F(z) = \sum_{n\equiv \frac r3\spmod 8} a(n) q^\frac n8 \in S_{\lambda+\frac12}(N,\psi\nu^r)
	\end{equation}
	and if  $\lambda=1$ suppose further that $F\in S_{3/2}^{\rm{c}}(N,\psi\nu^r)$.
	Define coefficients  $b(n)$ by the relation
\begin{equation}\label{eq:stdef3}
		\sum_{n=1}^\infty \frac{b(n)}{n^s} = L\(s-\lambda+1,\psi \ptfrac \bullet t\) \sum_{n=1}^\infty \pfrac{-4}n\frac{a(tn^2)}{n^s}.
	\end{equation}
	Then
	\begin{equation}
		\Sh_t(F) := \sum_{n=1}^\infty b(n)q^n \in S_{2\lambda}^{\new 2}(2N,\psi^2, \ep_{2, r, \psi}).
	\end{equation}
Furthermore,  we have
\begin{equation}\label{equivariance3midr}
\Sh_t(T_{p^2}F)=\pmfrac{-4}pT_p\Sh_t(F)\ \ \ \ \text{for each prime $p\geq 3$}.
\end{equation}
\end{theorem}
\begin{remark}
In this case we have $\Sh_t(F)=0$ unless $t \equiv  r/3 \pmod{8}$.  
The precise relationship between $\Sh_t$ and the standard Shimura lift $\Shl_t$ in both cases is described precisely in Section~\ref{subsec:compare}.
\end{remark}

As an application of these theorems we prove that quadratic congruences of a particular type hold for modular forms with the eta-multiplier in a wide range of spaces.  These congruences are motivated by some old examples of Atkin for the partition function \cite{Atkin_mult} which are described in \eqref{eq:atkincong} below.
The fact that the  lifts are new at $2$ is crucial to the arithmetic techniques which  we employ, which generalize recent results of the first author with Allen and Tang \cite{Ahlgren-Allen-Tang}.

For the application we assume that $\psi$ is real and that $\ell\geq 5$ is prime.
Our first result on congruences relies on the assumption that the pair $(2\lambda,\ell)$ is \emph{suitable} for the triple $(N,\psi,r)$.   This is a technical hypothesis on the mod $\ell$ reductions $\bar\rho_f$ of the $\ell$-adic Galois representations $\rho_f$ attached to newforms $f$.
However, we will see in Section~\ref{subsec:suitable} that $(2\lambda,\ell)$ is suitable for 
any triple $(N,\psi,r)$ if 
\begin{equation}\label{eq:suitablenumeric}
    \ell > 10\lambda-4 \ \ \ \text{and}\ \ \ 2^{2\lambda-1} \not \equiv 2^{\pm 1} \pmod{\ell}.
\end{equation}

\begin{definition}
\label{def:suitable}
Suppose that $\ell \geq 5$ is prime and that $r$ is an odd integer. Suppose that $N$ is a squarefree, odd, positive integer with $\ell \nmid N$, and $3 \nmid N$ if $3 \nmid r$. Suppose that $\psi$ is a real character modulo $N$, and  let $k$ be a positive even integer.
If $(r,6)=1$, then 
we say that the pair $(k,\ell)$ is \emph{suitable} for the triple $(N,\psi,r)$ if
for every normalized Hecke eigenform $f \in S^{\new 2,3}_{k}(6N,\ep_{2,r,\psi},\ep_{3,r,\psi})$, the image of $\bar{\rho}_f$ contains a conjugate of 
$\SL_2(\mathbb{F}_\ell)$. If $(r,6)=3$, then we make a similar definition for normalized Hecke eigenforms in $S^{\new 2}_{k}(2N,\ep_{2,r,\psi})$.
\end{definition}

  To make the statement of the next two results uniform, in the case when $3\mid r$ we choose to express the Fourier expansion \eqref{eq:Fdef3} by a change of variables in the form \eqref{eq:Fdeff}.
Let  $S_{\lambda+{1}/2}(N,\psi\nu^{r})_\ell  \subseteq S_{\lambda+{1}/2}(N,\psi\nu^{r})$ denote the subset of forms whose coefficients are algebraic numbers which are integral at all primes above $\ell$.

\begin{theorem}\label{thm:cong1}
Suppose that  $\ell \geq 5$ is prime and that $r$ is odd.
Suppose that $m$ and $\lambda$ are positive integers.
Let $N$ be a squarefree, odd positive integer such that $\ell \nmid N$, and $3 \nmid N$ if $3 \nmid r$. 
Let $\psi$ be a real character modulo $N$. 
Suppose that 
\begin{equation}
F(z) =\displaystyle \sum_{n \equiv r \spmod{24}} a(n)q^{\frac{n}{24}} \in S_{\lambda+\frac{1}2}(N, \psi\nu^{r})_\ell
\end{equation}
with $(2\lambda,\ell)$ suitable for $(N,\psi,r)$, and if $\lambda=1$, suppose further that $F \in S^{c}_{{3}/2}(N,\psi\nu^r)$. 
Then there is a positive density set $S$ of primes such that if $p \in S$, then $p \equiv 1 \pmod{\ell^m} $ and
\begin{equation}
a(p^2n) \equiv 0 \pmod{\ell^m}  \ \ \ \ \text{ if } \ \ \ \(\frac{n}{p}\)=\begin{cases} 
\(\frac{-1}{p}\)^{\frac{r-1}{2}}\psi(p) & \text{if } 3 \nmid r, \\
\(\frac{-3}{p}\)\(\frac{-1}{p}\)^{\frac{r-1}{2}}\psi(p) & \text{if }  3\mid r.
\end{cases}
\end{equation}
\end{theorem}

\begin{remark}
In the above theorem and those which follow, our definition of density is that of natural density.  The set $S$ depends on $F$, $\ell$ and $m$.
\end{remark}
\begin{remark}
It would also be possible to prove an analogue of \cite[~Theorem $1.2$]{RJD} by modifying the proof of Theorem~\ref{3.9 analogue} below. 
See the remark at the end of Section $7$.
\end{remark}

Our next result on congruences does not rely on the hypothesis of suitability.
 \begin{theorem}\label{thm:cong2}
 Suppose that $\ell \geq 5$ is prime and that $r$ is odd.
 Let $m$ and $\lambda$ be positive integers.
 Let $N$ be a squarefree, odd positive integer such that $\ell \nmid N$, and $3 \nmid N$ if $3 \nmid r$.
 Let $\psi$ be a real character modulo $N$. Suppose that there exists  $a \in \Z$ with the property that
 \begin{equation}\label{Hasse}
 2^{a} \equiv -2 \pmod\ell . 
 \end{equation}
   Let 
   \begin{equation}
   F(z) =\displaystyle \sum_{n \equiv r \spmod{24}} a(n)q^{\frac{n}{24}} \in S_{\lambda+\frac{1}2}(N, \psi\nu^{r})_\ell,
   \end{equation}
and if $\lambda=1$, suppose further that $F \in S^{c}_{{3}/2}(N,\psi\nu^r)$.
 Then there is a positive density set $S$ of primes such that if $p \in S$ then $p \equiv -2 \pmod{\ell^{m}}$ and for some $\ep_{p} \in \{\pm 1\}$ we have
 \begin{equation}
 a(p^2n) \equiv 0 \pmod{\ell^{m}} \ \ \ \text{ if } \ \ \ \(\frac{n}{p}\)=\ep_{p}.
 \end{equation}
 \end{theorem}
\begin{remark}
The value of $\ep_{p}$ can be explicitly calculated using Theorem~\ref{Theorem 4.2 analogue} below.  By a result of Hasse \cite{Hasse},  the proportion of primes satisfying \eqref{Hasse} is $17/24$.
\end{remark}

As an example of an application of our main results, we consider congruences for colored generalized Frobenius partitions, which were introduced by 
Andrews \cite[\S 4]{Andrews}, and have been studied by many authors.
For a positive integer $m$, let $c\phi_{m}(n)$ be the number of generalized Frobenius partitions of $n$ with $m$ colors. 
By \cite[Theorem 5.2]{Andrews} we have (for $m\geq 2$)
\begin{equation}\label{eq:cphidef}
\sum c\phi_m\pmfrac{n+m}{24}q^\frac n{24}=\eta^{-m}(z)\sum_{n=0}^\infty r_m(n)q^n,
\end{equation}
where $r_m(n)$ is the number of representations of $n$ by the quadratic form
\begin{equation}\label{eq:r_Qdef}
    \sum_{i=1}^{m-1}x_i^2+\sum_{1\leq i<j\leq m-1}x_ix_j.
\end{equation}
By \cite[(5.15)]{Andrews}, $c\phi_1(n)$ agrees with the ordinary partition function $p(n)$.
Congruence properties of $c\phi_m$ have been studied by many people; see for example \cite{baruah-sarmah,cui-gu-huang,garvan-sellers,gu-wang-xia,lin,lovejoy-frob,petta,wang-zhang}.
Since the generating function  $\sum r_m(n)q^n$ is a holomorphic modular form of weight $(m-1)/2$ and level $m$ or $2m$,
it follows from the results of 
Treneer \cite{treneer_1} that if $\ell\geq 5$ is a prime with $\ell\nmid m$ and $j$ is a positive integer,  then there are infinitely many $Q$ giving rise to  congruences of the form
\begin{equation}\label{eq:treneerfrob}
c\phi_m\pmfrac{\ell^kQ^3 n+m}{24}\equiv 0\pmod {\ell^j} \ \ \ \text{if} \ \ \ (n, \ell Q)=1
\end{equation}
 where $k$ is sufficiently large (see \cite[Theorem~2.1]{Chan-Wang-Yang}, \cite[Lemma~1, Corollary~1]{jameson_wieczorek} for details).

For the partition function, Atkin \cite{Atkin_mult} found a number of examples of congruences
of the form 
\begin{equation}\label{eq:atkincong}
    p \pmfrac{\ell Q^2 n+1}{24}\equiv 0 \pmod{\ell}\ \ \ \ \text{if}\ \ \ \pmfrac n{Q}=\ep_Q,\\
\end{equation}
where  $\ell$ and $Q$ are distinct primes, $\ep_Q\in \{\pm1\}$,  and  $5\leq \ell\leq 31$.
In recent work of the first author with Allen and Tang \cite{Ahlgren-Allen-Tang} it is shown that for every prime $\ell\geq 5$, a positive proportion of primes $Q$ give rise to a congruence of the form \eqref{eq:atkincong}. 

Our main theorems open the door to proving congruences like Atkin's for the functions $c\phi_m$.  Since the constructions are somewhat involved,
we will develop this application fully in a forthcoming paper. In Section~\ref{sec:GFP} we give an extended example which illustrates the use of these  theorems in the case when $m=5$.  The simplest examples of the congruences which we obtain are
\begin{align}
 c\phi_{5}\pmfrac{13 \cdot97^2 n+5}{24} &\equiv 0 \pmod{13}\ \ \ \ \text{if}\ \ \ \pmfrac n{97}=-1,\\
  c\phi_{5}\pmfrac{13 \cdot103^2 n+5}{24} &\equiv 0 \pmod{13}\ \ \ \ \text{if}\ \ \ \pmfrac n{103}=1.
\end{align}
Note that selecting $n$ in appropriate residue classes  gives rise to many congruences of the form
\begin{equation}
    c\phi_5(\ell Q^3 n+\beta)\equiv 0\pmod {13};
\end{equation}
for example, choosing $n\equiv 199\pmod{24\cdot 97}$ in the first example above produces the congruence
\begin{equation}\label{eq:congex}
    c\phi_5(13\cdot97^3\, n+1014212)\equiv 0\pmod{13}.
\end{equation}
In the last section we give many other such examples; these can easily be checked numerically since $c\phi_5$ can be expressed in terms of the partition function using \cite[(1.13)]{Chan-Wang-Yang}.

We close the Introduction with a brief outline of the paper and a sketch of our methods.  Section~\ref{sec:background} contains background results on the various sorts of modular forms which we consider.
To prove Theorem~\ref{thm:shimura-lift} in the case $r=t=1$, we begin by constructing in Section~\ref{sec:theta-kernels} a two-variable theta function $\vartheta(z,w)$ from a lattice $L$ of rank 3 and a ternary quadratic form $Q$ (this construction follows the outline of Niwa and Cipra).
The function $\vartheta(z,w)$ transforms with weight $\lambda+ 1/2$ on $\Gamma_0(N)$ in the $z$-variable and with weight $2\lambda$ on $\Gamma_0(6N)$ in the $w$-variable.
We prove directly that $\vartheta(z,w)$ behaves nicely with respect to the Atkin-Lehner involutions $W_2$ and $W_3$ and the Fricke involution.
After checking analytic behavior, we see that the function
\begin{equation}
    \Phi(w)=\int_{\Gamma_0(N)\backslash \H} v^{\lambda+\frac 12}F(z) \overline{\vartheta(z,w)} \, \frac{dudv}{v^2}
\end{equation}
gives the lift $\Sh_1(F)$.    

A lengthy but reasonably straightforward calculation in Section~\ref{sec:lift} gives the  Fourier expansion of $\Phi(w)$.  We use various operators to prove the remaining assertions in Theorem~\ref{thm:shimura-lift} and to deduce the theorem for the remaining values of $r$ and $t$.
In particular, the Hecke equivariance of the lift at primes $\geq 5$ follows from the Fourier expansion, while
the behavior at the primes $2$ and $3$ is inherited from that of $\vartheta(z,w)$.  The proof of Theorem~\ref{thm:shimura-lift-r=3} parallels that of Theorem~\ref{thm:shimura-lift}; in Section~\ref{sec:3midr} we describe the construction of the theta kernel and give a sketch of the remainder of the proof since the details are similar.

 The arguments in Section~\ref{sec:quadcong} used to prove Theorems~\ref{thm:cong1} and \ref{thm:cong2} are Galois-theoretic.  We begin by showing that the condition of suitability is satisfied for most spaces (in particular, justifying the assertion \eqref{eq:suitablenumeric}).
 We require modifications of the arguments of \cite{Ahlgren-Allen-Tang}; the main technical results are Theorems~\ref{3.9 analogue} (which relies on suitability) and \ref{Theorem 4.2 analogue} (which does not); these give large sets of primes for which the Hecke operators act diagonally with prescribed eigenvalues on newforms in the relevant spaces  modulo arbitrary powers of a given prime $\ell \geq 5$. Filtered through the maps $\Sh_t$, these eigenvalues give the congruences described in the theorems.  
 Finally, in Section~\ref{sec:GFP} we give an extended example which produces the congruences for $c\phi_5$ described above.

 \subsection*{Acknowledgments}
We thank the referees for many helpful comments.

\section{Examples}\label{sec:Examples}
\subsection*{Example 1}
The space $S_4(6)=S_4^{\new 2, 3}(6, +1, +1)$ is one-dimensional, spanned by the newform
\begin{equation}
f(z)=\eta^2(z)\eta^2(2z)\eta^2(3z)\eta^2(6z)=\sum a(n)q^n= q-2 q^2-3 q^3+4 q^4+6 q^5+6 q^6+\cdots.
\end{equation}
Let 
\begin{equation}
F(z)=\eta^5(z)=q^\frac5{24}-5 q^\frac{29}{24}+5 q^\frac{53}{24}+10 q^\frac{77}{24}-15 q^\frac{101}{24}+\cdots\in S_\frac52(1, \nu^5).
\end{equation}
It follows from Theorem~\ref{thm:shimura-lift} that each lift 
$\Sh_t(F)$ is a constant multiple of $f$.  By \eqref{equivariance3nmidr} we conclude that for every square-free $t$ and each prime $p\geq 5$ we have 
\begin{equation*}
\Sh_t\(T_{p^2}F-\pmfrac{12}pa(p)F\)=0.
\end{equation*}
It follows from a standard argument (see e.g. \cite[(2.14)]{ahlgren-beckwith-raum} that 
\begin{equation}\label{eq:ex1}
T_{p^2}F=\pmfrac{12}pa(p)F.
\end{equation}
We remark that if 
\begin{equation}G(z)=\frac{\eta^3(2z)\eta^2(3z)\eta^2(12z)}{\eta^2(6z)}=q-3 q^3-2 q^4+6 q^6+6 q^7-3 q^9+\cdots\in S_\frac52(12, \nu_\theta^5),\end{equation}
then  each lift  $\Shl_t(G)$ is also a constant multiple of $f$.

\subsection*{Example 2}
Consider the modular form $f(z)\in S_2(14)=S_2^{\new 2}(14, +1)$ defined by 
\begin{equation}
f(z)=\eta(z)\eta(2z)\eta(7z)\eta(14z)=\sum a(n)q^n=q - q^2 - 2 q^3 + q^4 + 2 q^6 + q^7+\cdots .
\end{equation}
Define (see Corollary~\ref{cor:v-op} to compute the multipliers)
\begin{equation}
F_1(z)=\eta(7z)\eta^2(z)\in S_\frac32\(7, \ptfrac\bullet 7\nu^9\), \qquad
F_2(z)=\eta(7z)^2\eta(z)\in S_\frac32\(7, \nu^{15}\).
\end{equation}
Arguing as in the first example, we find that 
  each lift $\Sh_t(F_i)$ is a constant multiple of $f$, and that for  $p\geq 3$  we have
\begin{equation}\label{eq:ex2}
T_{p^2}F_i=\pmfrac{-4}pa(p)F_i.
\end{equation}

\subsection*{Example 3}  Let $f\in S_2^{\new 2}(26, -1)$ be
\cite[\href{https://www.lmfdb.org/ModularForm/GL2/Q/holomorphic/26/2/a/b/}{newform 26.2.a.b}]{lmfdb}; we have 
\begin{equation}
f=\sum a(n)q^n=q + q^2 - 3q^3 + q^4 - q^5 - 3q^6 + q^7+\cdots.
\end{equation}
Then the conclusions of Example 2 hold with 
\begin{equation}
F_1(z)=\eta(13z)\eta^2(z)+\tfrac{13}7\eta(13z)^3\in S_\frac32\(13, \ptfrac\bullet {13}\nu^{15}\), \ 
F_2(z)=7\eta(13z)^2\eta(z)+\eta^3(z)\in S_\frac32\(13, \nu^{3}\).
\end{equation}

\subsection*{Example 4}  Finally,  let $f\in S_2^{\new 2, 3}(66, -1, -1)$ be
\cite[\href{https://www.lmfdb.org/ModularForm/GL2/Q/holomorphic/66/2/a/c/}{newform 66.2.a.c}]{lmfdb}; we have
\begin{equation}
f=\sum a(n)q^n=q + q^2 + q^3 + q^4 - 4q^5 + q^6-2q^7+\cdots.
\end{equation}
We find that the two forms
\begin{equation}
F_1(z)=\eta(11z)\eta^2(z)\in S_\frac32\(11, \ptfrac\bullet {11}\nu^{13}\), \qquad
F_2(z)=\eta(11z)^2\eta(z)\in S_\frac32\(11, \nu^{23}\)
\end{equation}
each lift to $f$ and satisfy the relationship \eqref{eq:ex1}.

We note  in each of the last three examples that the modular forms $F_1$ and $F_2$ are (up to a constant multiple) interchanged by the Fricke involution.

\section{Background and preliminaries}\label{sec:background}

If $f$ is a function on the upper half-plane $\H$,  $k\in \frac12\Z$,  and $\gamma=\pmatrix abcd\in \GL_2^+(\R)$, 
we define
\begin{equation}
\(f\|_k\gamma\)(z)=(\det\gamma)^\frac k2(cz+d)^{-k}f(\gamma z)
\end{equation} and 
\begin{equation}\label{eq:fslashstar}
\(f\|^*_k\gamma\)(z)=(\det\gamma)^\frac k2\bar{(cz+d)^{-k}}f(\gamma z).
\end{equation}
If $N\in \N$ and $\omega$ is a multiplier on $\Gamma_0(N)$, then we denote by 
$\mathcal M_k(N, \omega)$ the $\C$-vector space of functions on $\H$  which  satisfy the transformation law
\begin{equation}
f\|_k\gamma=\omega(\gamma)f \quad \text{ for all } \gamma \in \Gamma_0(N).
\end{equation}
We denote by  $M_k(N, \omega)$ and  $S_k(N, \omega)$ the subspaces of holomorphic and cuspidal modular  forms, respectively.

If $f$ is a function on $\H$ and $m\in \N$ we define $f\sl V_m(z)=f(mz)$; in other words 
\begin{equation}\label{eq:vmdef}
 f \sl V_m=m^{-\frac{k}2} f\|_k\pMatrix m001.
 \end{equation}
We also define 
\begin{equation}\label{eq:umdef}
   f \sl U_m=m^{\frac k2-1} \sum_{v\spmod m} f\|_k\pMatrix 1v0m.
   \end{equation}
If $f$ is holomorphic and has period $1$, so that $f(z)=\sum a(n)q^n$, then we have
\begin{equation}
  f\sl U_m=\sum a(mn) q^n.
\end{equation}

\subsection{Integral weight modular forms}\label{sec:intweight}

Suppose that  $k\in \Z$, that $N\in \N$, and that $\chi$ is a Dirichlet character modulo $N$.    Define
\begin{equation}\label{eq:fricke_def}
H_N =\pMatrix 0{-1}N0.
\end{equation}
The Fricke involution $f\mapsto f\|_kH_N$ takes  $\mathcal{M}_k(N, \chi)$ to   $\mathcal{M}_k(N, \bar\chi)$.  For primes $p$, the map $f\mapsto f\|V_p$ takes   $\mathcal{M}_k(N, \chi)$ to   $\mathcal{M}_k(Np, \chi)$.  If $p\nmid N$ then the map 
$f\mapsto f\|U_p$ takes   $\mathcal{M}_k(N, \chi)$ to   $\mathcal{M}_k(Np, \chi)$, while if $p\mid N$ then the map preserves $\mathcal{M}_k(N, \chi)$.

 For each prime $p$ we have the Hecke operator $T_p: S_k(N, \chi)\to S_k(N, \chi)$ defined by 
\begin{equation}\label{eq:inthecke}
T_p=U_p+\chi(p)p^{k-1}V_p.
\end{equation}
If $p\mid N$ and $(p, N/p)=1$, then    the Atkin-Lehner matrix $W^N_p$ is any integral matrix with 
\begin{equation}\label{eq:atkinlehner}
W^N_p =\pMatrix{p\alpha }\delta{N\beta}{p }, \ \ \ \ \alpha,\beta,\delta \in \Z, \ \ \ \ \det(W^N_p)=p.
\end{equation}
If $\chi$ is defined modulo $N/p$ 
then the operator $\|_kW^N_p$ preserves the space $\mathcal{M}_k(N, \chi)$, and on this space is independent of the particular choice of matrix \cite[Lemma 2]{Li:1975}.  Note that we may  take $\delta=1$ in \eqref{eq:atkinlehner}.
Since scalar matrices act trivially under $\|_k$, it will sometimes be convenient to work with the scaled matrices
\begin{equation}\label{eq:fricke_def_1}
H_N=\pMatrix 0{-1/\sqrt{N}}{\sqrt N}0,
\end{equation}
\begin{equation}
W^N_p=\pMatrix{\sqrt p \, \alpha}{\delta/\sqrt p}{N\beta/\sqrt p}{\sqrt p }.
\end{equation}
All of the above statements  remain true with $\mathcal{M}_k$ replaced by $M_k$ or $S_k$.

If $p\mid N$ and $\chi$ is defined modulo $N/p$, then we denote by $S_k^{\new p}(N, \chi)$ the orthogonal complement (with respect to the Peterson inner product) of the subspace generated by 
$S_k\(N/p, \chi\)$ and $S_k\(N/p, \chi\)\|V_p$.
By \cite[Lemma 6]{Li:1975} and an argument as in the proof of \cite[Theorem~4]{Li:1975} we have 
\begin{equation}\label{eq:traceequiv}
f\in S_k^{\new p}(N, \chi) \iff 
\Tr^N_{N/p}f=\Tr^{N}_{N/p}\(f\sl_kH_{N}\)=0,
\end{equation}
where 
\begin{equation}
    \Tr^{N}_{N/p}: S_k\(N, \chi\)\to S_k\(N/p, \chi\)
\end{equation}
is the trace map defined by 
\begin{equation}
\Tr^N_{N/p}f=\sum f\|_kR_i
\end{equation}
where the $R_i\in \Gamma_1(N/p)$ are right coset representatives of $\Gamma_0(N)$ in $\Gamma_0(N/p)$. 
We record a standard lemma for convenience.
\begin{lemma}\label{lem:trace}
    Suppose that $p\mid N$ is a prime with $(p, N/p)=1$, that $\chi$ is a Dirichlet character modulo $N/p$, and that $f\in S_k(N, \chi)$.  Then 
    \begin{align*}\label{eq:trace_def}
\Tr^N_{N/p}f&= f+\bar\chi(p)p^{1-\frac k2}f\sl _k W_p^N\|U_p,
\\ \Tr^N_{N/p}\(f\sl_k H_N\)&= f\sl_kH_N+\chi(p)p^{1-\frac k2}f\sl_kH_{N}\sl_kW_p^N\|U_p.
\end{align*}
\end{lemma}

\begin{proof}  The second assertion follows from the first (recall that $f\big|_kH_N\in S_k(N, \bar\chi)).$ To prove the first, 
write $W_p^N=\pmatrix {p\alpha}1{N\beta}p$. 
The identity matrix together with the matrices
\begin{equation} S_j:=\tfrac1{p}W_p^N\pmatrix 1j0p=\pMatrix \alpha {1+j\alpha}{\frac{N\beta}p}{\frac{jN\beta}p+p}, \ \ \ 0\leq j< p
\end{equation}
are a set of right coset representatives for $\Gamma_0(N)$ in $\Gamma_0(N/p)$.
Choose a matrix $\gamma=\pmatrix a b c d\in \Gamma_0(N)$ with $a\equiv p\pmod{N/p}$.
Then the identity matrix together with the matrices  $\gamma S_j$, $0\leq j<p$ form a set of representatives $R_i$ of the required form.  The lemma follows from \eqref{eq:umdef}.
\end{proof}
\begin{corollary}\label{cor:upeigen}
     Suppose that $p\mid N$ is a prime with $(p, N/p)=1$, that $\chi$ is a Dirichlet character modulo $N/p$, and that $f\in S_k^{\new p}(N, \chi, \ep_p)$ (where $\ep_p$ denotes the $W_{p}^N$ eigenvalue).  Then 
    \begin{equation}
        f\sl U_p=-\ep_pp^{\frac k2-1} f.
    \end{equation}
    In particular if $f=\sum a(n)q^n$ is a newform then $a(p)=-\ep_pp^{\frac k2-1}$.
\end{corollary}
\begin{proof}
From the lemma and \eqref{eq:traceequiv} we have 
$f\sl U_p=-\bar\ep_p\chi(p)p^{\frac k2-1} f$.
The claim follows since $\ep_p^2=\chi(p)$ (see e.g. \cite[Prop. 13.3.14]{cohen}).
\end{proof}
\subsection{Modular forms for the theta-multiplier}\label{sec:thetamult}
The standard theta function is given by 
 \begin{equation}
  \theta(z) = \sum_{n=-\infty}^\infty q^{n^2}.
  \end{equation}
  This is a modular form of weight $1/2$ on $\Gamma_0(4)$ with  multiplier $\nu_\theta$  defined by 
\begin{equation}
 \theta\|_\frac12\gamma=\nu_\theta(\gamma)\theta,\ \ \ \gamma\in \Gamma_0(4).
  \end{equation}
  If $\gamma=\pmatrix abcd\in \Gamma_0(4)$ then we have 
\begin{equation}\label{eq:nutheta}
\nu_\theta(\gamma)=  \pmfrac{c}{d} \ep_d^{-1},
\end{equation}
where
\begin{equation}
\ep_d = \begin{cases} 
1\ \ &\text{if}\ \ \   d\equiv 1 \pmod{4},\\
i\ \ \ &\text{if}  \ \ \ d \equiv 3 \pmod{4}.
\end{cases}
\end{equation}
For odd values of $d$, $d_1$, $d_2$ we recall the formulas
\begin{alignat}2
\label{eq:oneminusd-epd1d2}
 e\pmfrac{1-d}8&= \pmfrac2d\ep_d
\quad\text{and}\quad
  \ep_{d_1d_2} &= \ep_{d_1}\ep_{d_2}(-1)^{\frac{d_1-1}2\frac{d_2-1}2},
\end{alignat}
where $e(x) = e^{2\pi ix}$.

The spaces of modular forms of half integral weight in the sense of Shimura \cite{shimura} are 
\begin{equation}
M_{\lambda+\frac12}(N, \psi\nu_\theta^{2\lambda+1}),
\end{equation}
where $4\mid N$ and $\psi$ is a Dirichlet character modulo $N$.
On these spaces there are Hecke operators $T_{p^2}^{\rm s}$ for primes $p$. For $F=\sum a(n)q^n\in M_{\lambda+1/2}(N, \psi\nu_\theta^{2\lambda+1})$ we have the explicit description
\begin{equation}\label{eq:heckedefs}
T_{p^2}^{\rm s} F=\sum \(a(p^2n)+\ptfrac{-1}p^\lambda\ptfrac np \psi(p)p^{\lambda-1}a(n)+\psi^2(p)p^{2\lambda-1}a\ptfrac n{p^2}\)q^n.
\end{equation}

We recall the definition of the standard Shimura lift.  Suppose that $F=\sum a(n)q^n\in S_{\lambda+1/2}(N, \psi\nu_\theta^{2\lambda+1})$,
where $\lambda\geq 1$;  if $\lambda=1$ suppose further that $F$ is in the orthogonal complement of the space spanned by single variable theta series.
If $t$ is a positive squarefree integer, Shimura's lift is given by  $\Shl_t(F)=\sum c(n)q^n$, where the coefficients $c(n)$ are given by 
\begin{equation}\label{eq:shimliftold}
\sum_{n=1}^\infty \frac{c(n)}{n^s}=L\(s-\lambda+1, \psi\ptfrac{-1}\bullet^\lambda\ptfrac t\bullet\)\sum_{n=1}^\infty \frac{a(tn^2)}{n^s}.
\end{equation}
After the work of Shimura, Niwa, and Cipra \cite{shimura, niwa, cipra} we have $\Shl_t(F)\in S_{2\lambda}\(N/2, \psi^2\)$.
Moreover, the lift is  equivariant with respect to the Hecke operators $T_p$ and $T_{p^2}^{\rm s}$.

\subsection{Modular forms for the eta-multiplier}
The Dedekind eta function is given by 
 \begin{equation}
  \eta(z)=q^\frac1{24}\prod_{n=1}^\infty(1-q^n).
  \end{equation}
This is a modular form of weight $1/2$ on $\SL_2(\Z)$, and the eta-multiplier $\nu$ is defined by 
\begin{equation}\label{eq:etamult}
\eta\|_\frac12\gamma=\nu(\gamma)\eta, \qquad \gamma\in \SL_2(\Z).
\end{equation}
Note that we have $\nu(\gamma)^{24}=1$ for all $\gamma$.
We have the  explicit formulas \cite[\S 4.1]{knopp} for $c>0$:
\begin{equation}
\begin{aligned} \label{eq:nueta}
  &\nu(\gamma) = 
  \begin{dcases}
    \( \mfrac dc \) \, e\(\mfrac 1{24} \( (a+d)c-bd(c^2-1)-3c \) \),
  & \text{ if $c$ is odd}, \\
    \(\mfrac cd \) \, e\(\mfrac 1{24} \((a+d)c-bd(c^2-1)+3d-3-3cd\)\),
  & \text{ if $c$ is even,}
  \end{dcases}\\
  &\nu(-\gamma)=i\nu(\gamma).
  \end{aligned}
\end{equation}

Let $\psi$ be a Dirichlet character defined modulo $N$, and $r\in \Z$.
We will be concerned with the spaces 
$M_{\lambda+1/2}\(N, \psi \nu^r\)$ where $\lambda\in \Z_{\geq0}$.
When $\psi$ is trivial we omit it from the notation.
We recall  \cite[(2.6)]{ahlgren-beckwith-raum} that
\begin{equation}\label{eq:krcond}
M_{\lambda+\frac12}\(N, \psi \nu^r\)=\{0\}
\quad\text{unless}\quad
 \psi(-1)\equiv r-2\lambda \pmod{4}.
\end{equation}
In particular these spaces are trivial unless $r$ is odd.
This condition may be written in the  form 
\begin{equation}\label{eq:krcond1}
M_{\lambda+\frac12}\(N, \psi \nu^r\)=\{0\}
\quad\text{unless}\quad
\psi(-1)=\pmfrac{-1}r(-1)^\lambda.
\end{equation}

Each  $F\in M_{\lambda+1/2}\(N, \psi \nu^r\)$ has a Fourier expansion of the form
\begin{equation}F(z)=\sum_{n\equiv r\spmod{24}} a(n)q^\frac n{24};
\end{equation}
if $(r, 6)=3$ it will typically be more convenient to represent this expansion in the form
\begin{equation}F(z)=\sum_{n\equiv \frac r3\spmod{8}} a(n)q^\frac n8.
\end{equation}
The next lemma describes a connection between these two multipliers.
\begin{lemma}\label{lem:eta_to_theta}
If $(r, 6)=1$ then 
\begin{equation}\label{eq:etamultv24final}
F \in \mathcal M_{\lambda+\frac12}\(N, \psi\nu^r\)
\implies
F\sl  V_{24} \in \mathcal M_{\lambda+\frac12}\(576N,  \psi\ptfrac{-1}\bullet^{\lambda+\frac{r-1}2} \ptfrac{12}\bullet\nu_{\theta}^{2\lambda+1} \).
\end{equation}
If $(r, 6)=3$ then
\begin{equation}\label{eq:etamultv8final}
F \in \mathcal M_{\lambda+\frac12}\(N, \psi\nu^r\)
\implies
F\sl  V_{8} \in \mathcal M_{\lambda+\frac12}\(64N,  \psi\ptfrac{-1}\bullet^{\lambda+\frac{r-1}2} \nu_{\theta}^{2\lambda+1} \).
\end{equation}
\end{lemma}

\begin{proof}
Suppose that $(r, 6)=1$.  After a computation using  \eqref{eq:nueta}, \eqref{eq:nutheta} and  \eqref{eq:oneminusd-epd1d2}, we find that
\begin{equation}\label{eq:etamultv24}
F \in \mathcal M_{\lambda+\frac12}\(N, \psi\nu^r\)
\implies
F\sl  V_{24} \in \mathcal M_{\lambda+\frac12}\(576N,  \psi\ptfrac{12}\bullet \nu_{\theta}^r \).
\end{equation}
From \eqref{eq:krcond}  we see that 
\begin{equation}
\nu_\theta^r=\nu_\theta^{2\lambda+1}\nu_\theta^{\pfrac{-1}r(-1)^\lambda-1}=\nu_\theta^{2\lambda+1}\ptfrac{-1}\bullet^\frac{\pfrac{-1}r(-1)^\lambda-1}2.
\end{equation}
For all $\lambda$ we have
\begin{equation}\label{eq:lambdapar}
\mfrac{\pfrac{-1}r(-1)^\lambda-1}2\equiv \lambda+\mfrac{r-1}2\pmod 2,
\end{equation}
from which 
\begin{equation}\label{eq:thetamultr}
\nu_\theta^r=\nu_\theta^{2\lambda+1}\ptfrac{-1}\bullet^{\lambda+\frac{r-1}2}.
\end{equation}
The lemma follows in the case $(r, 6)=1$.

If $(r, 6)=3$ then a computation analogous to that which establishes \eqref{eq:etamultv24} shows that
\begin{equation}
F \in \mathcal M_{\lambda+\frac12}\(N, \psi\nu^r\)
\implies
F\sl  V_{8} \in \mathcal M_{\lambda+\frac12}\(64N,  \psi \nu_{\theta}^r \).
\end{equation}
The lemma follows from the relationship \eqref{eq:thetamultr}.
\end{proof}
 Hecke operators can be defined on $M_{\lambda+1/2}\(N, \psi\nu^r\)$ using the  definition \eqref{eq:heckedefs} on spaces with the theta multiplier together with \eqref{eq:etamultv24} and \eqref{eq:etamultv8final} (see e.g. \cite[Prop. 11]{Yang} for the case $N=1$).
Suppose that  $(r, 6)=1$, that  $p\geq 5$ is prime, and that  
\begin{equation}
		F(z) = \sum_{n\equiv r\spmod {24}} a(n) q^\frac n{24} \in  M_{\lambda+\frac12}(N,\psi\nu^r).
	\end{equation}
Then the action of $T_{p^2}$ is given by
\begin{equation}\label{eq:heckedef24}
T_{p^2}F=\sum_{n\equiv r(24)} \(a(p^2n)+\pmfrac{-1}p^\frac{r-1}2\pmfrac{12n}p\psi(p)p^{\lambda-1}a(n)+\psi^2(p)p^{2\lambda-1}a\pmfrac n{p^2}\)q^\frac n{24}.
\end{equation}
When   $(r, 6)=3$, $p\geq 3$ is prime, and   
\begin{equation}
		F(z) = \sum_{n\equiv \frac r3\spmod {8}} a(n) q^\frac n{8} \in  M_{\lambda+\frac12}(N,\psi\nu^r),
	\end{equation}
then the action is given by 
\begin{equation}\label{eq:heckedef8}
T_{p^2}F=\sum_{n\equiv \frac r3\spmod {8}} \(a(p^2n)+\pmfrac{-1}p^\frac{r-1}2\pmfrac{n}p\psi(p)p^{\lambda-1}a(n)+\psi^2(p)p^{2\lambda-1}a\pmfrac n{p^2}\)q^\frac n{8}.
\end{equation}
These operators preserve the spaces of cusp forms.
\begin{remark} If $(r, 6)=3$,
     then the definitions \eqref{eq:heckedef24} and \eqref{eq:heckedef8} disagree only when $p=3$.
\end{remark}

\begin{lemma} \label{lem:nu-v-t}
	Suppose that $r$ and $t$ are odd, and that $3\nmid t$ if $3\nmid r$.  
	If $\pmatrix abcd\in \Gamma_0(t)$ then
	\begin{equation}
		\nu^r\left(\pMatrix{a}{tb}{c/t}{d}\right) = \pfrac dt \nu^{rt} \left(\pMatrix{a}{b}{c}{d}\right).
	\end{equation}
\end{lemma}
\begin{proof}
	By the last identity in \eqref{eq:nueta} we may assume that $c>0$.  The lemma can then be checked using the explicit formulas in \eqref{eq:nueta}.  In the case when $c$ is even the  computation relies on    quadratic reciprocity in the form
	\begin{equation}
		\pfrac td e\left(\frac{(1-t)(d-1)}{8}\right) = \pfrac dt.
	\end{equation}
\end{proof}

\begin{corollary} \label{cor:v-op}
	Suppose that $r$ and $t$ are odd, and that $3\nmid t$ if $3\nmid r$.   If $F\in \mathcal{M}_{\lambda+1/2}(N,\psi \nu^r)$ then $F\sl V_t \in \mathcal{M}_{\lambda+1/2}(Nt,\psi \ptfrac\bullet t \nu^{rt})$.
\end{corollary}

Finally we record some technical results which will be important later.
As in \cite[\S 3.4]{serre-stark} we define an involution which acts on functions on $\H$ for $\lambda\in \Z$:
\begin{equation}\label{eq:fricke_half}
( F\sl \mathcal{W}_{N, {\lambda+\frac12}})(z):=N^{-\frac\lambda2-\frac14}(-iz)^{-\lambda-\frac12}F(-1/Nz).
\end{equation}
If $\gamma=\pmatrix abcd\in \Gamma_0(N)$ define 
\begin{equation}
    \gamma'=\pmatrix d{-c/N}{-bN}a=H_N\gamma H_N^{-1},
\end{equation}
and define $\ep_\gamma\in \{\pm 1\}$ as follows: $\ep_\gamma=1$  if and only if any of the following  is satisfied:
\begin{enumerate}
    \item $a\geq 0$.
    \item $a<0$ and $bc<0$.
    \item $b=0$, $c\geq 0$, and $a=d=-1$.
    \item $c=0$, $b\leq 0$, and $a=d=-1$.
\end{enumerate}

\begin{lemma} \label{lem:nu-w-N}
	If  $\gamma\in \Gamma_0(N)$ and $F$ is a function on $\H$ then 
 \begin{equation}
     F\sl \mathcal{W}_{N, \lambda+\frac12}\sl_{\lambda+\frac12}\gamma=\ep_\gamma F\sl_{\lambda+\frac12}\gamma'\sl \mathcal{W}_{N, \lambda+\frac12}.
 \end{equation}
\end{lemma}
\begin{proof}
	For the proof the following facts are useful:
\begin{equation*}
\pmfrac zw^\frac12=\mfrac{z^{\frac12}}{w^{\frac12}}, \ \ \ \ \  (-z)^\frac12=-iz^\frac 12\ \ \ \text{for $z$, $w\in \H$}.
\end{equation*}
Computing each side and using the first fact, we see that it suffices to prove that 
\begin{equation}\label{eq:branchpain}
    \pmfrac{az+b}{cz+d}^\frac12(cz+d)^\frac12=\ep_\gamma z^\frac12\pmfrac{az+b}z^\frac12.
\end{equation}
This is established with a straightforward but tedious calculation.  For example, if $a<0$, $b>0$ and $c<0$, we find that the left side of \eqref{eq:branchpain} is 
$-i(-az-b)^\frac12$, and that 
\begin{equation*}
    z^\frac12\pmfrac{az+b}z^\frac12=-i z^\frac12\pmfrac{-az-b}z^\frac12=-i(-az-b)^\frac12,
\end{equation*}
from which  $\ep_\gamma=1$. The other cases are similar and we omit the details.
\end{proof}

Let $\nu_N$ be the multiplier on $\Gamma_0(N)$  associated to $\eta(Nz)$.  
\begin{corollary}
    \label{cor:eta_fricke}
    If $\psi$ is a Dirichlet character modulo $N$ and $F\in \mathcal{M}_{\lambda+1/2}\(N, \bar\psi\nu_N\)$ then 
    \[F\sl \mathcal{W}_{N, \lambda+\frac 12}\in \mathcal{M}_{\lambda+\frac12}\(N, \psi\nu\).\]
\end{corollary}  
\begin{proof}
    Applying Lemma~\ref{lem:nu-w-N} with $F=\eta$ and using the fact that $\eta\sl\mathcal{W}_{N, 1/2}=N^\frac14\eta(Nz)$ shows that for $\gamma\in \Gamma_0(N)$ we have $\nu_N(\gamma)=\ep_\gamma\nu(\gamma')$.
    Since $\gamma''=\gamma$ it follows that $\nu_N(\gamma')=\ep_{\gamma'}\nu(\gamma)$.
   Applying the lemma with $F\in \mathcal{M}_{\lambda+1/2}\(N, \nu_N\bar\psi\)$ shows that for $\gamma\in \Gamma_0(N)$ we have 
 \begin{equation*}
      F\sl \mathcal{W}_{N, \lambda+\frac12}\sl_{\gamma+\frac12}\gamma=\ep_\gamma \nu_N(\gamma')\bar\psi(\gamma')F\sl \mathcal{W}_{N, \lambda+\frac12}.
 \end{equation*}
 Note that $\bar\psi(\gamma')=\psi(\gamma)$.
 It can be checked from the definition that $\ep_\gamma=\ep_{\gamma'}$, and
 the result follows from these facts.
\end{proof}

\subsection{Comparison of the lifts}\label{subsec:compare}
We prove a proposition describing the relationship between the lifts  $\Sh_t$  and $\Shl_t$.
\begin{proposition}\label{prop:shcompare}
Suppose that $(r, 6)=1$, that $t$ is a squarefree positive integer, and that
$F=\sum a(n)q^n \in  S_{\lambda+1/2}(N,\psi\nu^r)$ and $\Sh_t(F)=\sum b(n)q^n$ are as in \eqref{eq:Fdeff} and \eqref{eq:stdef}.
Let $\Shl_t(F\sl V_{24})=\sum c(n)q^n$ be the usual Shimura lift defined in \eqref{eq:shimliftold}.
Then for all $n$ we have 
\begin{equation}
c(n)= 
\pmfrac{12}nb(n).
\end{equation}

Suppose that $(r, 6)=3$, that $t$ is a squarefree positive integer,
 and that
$F=\sum a(n)q^n \in  S_{\lambda+1/2}(N,\psi\nu^r)$ and $\Sh_t(F)=\sum b(n)q^n$ are as in \eqref{eq:Fdef3} and \eqref{eq:stdef3}.
Let $\Shl_t(F\sl V_{8})=\sum c(n)q^n$.
Then for all $n$ we have 
\begin{equation}
c(n)= 
\pmfrac{-4}nb(n).
\end{equation}

\end{proposition}

\begin{proof}  
If $(r, 6)=1$ we may assume that $t\equiv r\pmod{24}$ (otherwise both lifts are zero).  Using Lemma~\ref{lem:eta_to_theta} and quadratic reciprocity we see that the coefficients of $\Shl_t(F\sl V_{24})$ are given by 

\begin{equation}
\sum \frac{c(n)}{n^s}=L\(s-\lambda+1, \psi\ptfrac{12}\bullet\ptfrac \bullet t\)
\sum\frac{a(tn^2)}{n^s}.
\end{equation}
The claim follows from comparing this with the definition of $b(n)$ (recall  that $a(n)=0$ if $(n, 6)\neq 1$).

If $(r, 6)=3$ and $t\equiv  r/3\pmod 8$,  then the coefficients of $\Shl_t(F\sl V_8)$ are given by 
\begin{equation}
\sum \frac{c(n)}{n^s}=L\(s-\lambda+1,\psi\ptfrac{-4}\bullet^\frac{1-r}2\ptfrac t\bullet\)
\sum\frac{a(tn^2)}{n^s}.
\end{equation}
Note that $\pfrac{-4}\bullet^\frac{1-r}2\pfrac t\bullet=\pfrac{-4}\bullet\pfrac\bullet t$.
The proposition follows in the same way.
\end{proof}

\section{Construction and properties of theta kernels}

\label{sec:theta-kernels}
In this section we modify the theta kernels of Niwa \cite{niwa} and Cipra \cite{cipra} to construct a theta function $\vartheta(z,w)$, $z,w\in \H$, with
\begin{equation}
    \vartheta(\cdot,w) \in \mathcal M_{\lambda+\frac 12}(N,\psi\nu) \quad \text{ and } \quad \vartheta(z,\cdot) \in \mathcal M_{2\lambda}(6N,\psi^2),
\end{equation}
and such that $\overline{\vartheta(z,\cdot)}$ is an eigenform of the operators $U_p$, $W_p^{6N}$, and $H_{6N}$.
Here $N$ and $\lambda$ are positive integers and $\psi$ is a Dirichlet character modulo $N$ with $\psi(-1)=(-1)^\lambda$ (see \eqref{eq:krcond1}).
In the next section we will use $\vartheta(z,w)$ to define the Shimura lift  and prove Theorem~\ref{thm:shimura-lift} in the case $r=t=1$.
The remaining cases are deduced by a separate argument.

Let $L$ be a lattice in $\R^n$ of rank $n$ and let $Q$ be an $n\times n$ symmetric matrix with rational entries and signature $(p,q)$, with $p+q=n$.
Define the bilinear form $\langle \cdot, \cdot \rangle:\R^n\times \R^n \to \R$ by $\langle x,y \rangle = x^T Q y$.
For $\sigma=\pmatrix abcd\in \SL_2(\R)$ the Weil representation $\sigma\mapsto r(\sigma)$ defined in \S 1 of \cite{cipra} (see also \S 1 of \cite{niwa}) acts on Schwartz functions $f:\R^n\to \C$ via
\begin{equation}
    \label{eq:weil-rep-def-integral}
	[r(\sigma)f](x) = 
	\begin{dcases}
		|a|^{\frac n2} e\left(\tfrac 12 ab\langle x,x \rangle\right) f(ax) & \text{ if }c=0, \\
		|\det Q|^{-\frac12}|c|^{-\frac n2}\int_{\R^n} e\left(\frac{a\langle x,x \rangle - 2\langle x,y \rangle + d \langle y,y \rangle}{2c}\right) f(y)\, dy & \text{ if }c\neq 0.
	\end{dcases}
\end{equation}
For $\mu\in \R$, a function $f:\R^n\to\C$ is said to have the weight $\mu$ spherical property if
\begin{equation} \label{eq:spherical-prop}
	r(\kappa(\phi))f = \ep(\kappa(\phi))^{p-q} e^{i\mu\phi}f
\end{equation}
where
\begin{equation} \label{eq:weil}
\kappa(\phi) = \pMatrix {\cos \phi}{\sin\phi}{-\sin\phi}{\cos\phi}
\qquad \text{ and } \qquad
	\ep(\sigma) = 
	\begin{cases}
		i^{\frac{\sgn c}2} & \text{ if }c\neq 0, \\
		i^{\frac{1-\sgn d}2} & \text{ if }c=0.
	\end{cases}
\end{equation}
Let $L^\ast = \left\{x\in \R^n : \langle x,y \rangle \in \Z \text{ for all }y\in L\right\}$ denote the dual lattice.
If $h\in L^\ast$ and $f$ is a Schwartz function with the weight $\mu$ spherical property, define
\begin{equation} \label{eq:theta-z-f-h-def}
	\theta(z,f,h) = v^{-\frac \mu2} \sum_{x\in L} [r(\sigma_z) f](x+h),
\end{equation}
where $z=u+iv$ and $\sigma_z\in \SL_2(\R)$ maps to $z$ under the map $\sigma\mapsto \sigma i$, that is,
\begin{equation}
	\sigma_z = \pMatrix{\sqrt v}{u/\sqrt v}{0}{1/\sqrt v}.
\end{equation}
Note that by \eqref{eq:weil-rep-def-integral} we have
\begin{equation} \label{eq:theta-z-f-h-expanded}
	\theta(z,f,h) = v^{\frac n4-\frac \mu2} \sum_{x\in L} e\left(\tfrac 12 u\langle x+h,x+h \rangle\right) f(\sqrt v(x+h)).
\end{equation}
Theorem~1.5 of \cite{cipra} (see also Corollary~0 of \cite{niwa}) gives the following transformation law for $\theta(z,f,h)$. 

\begin{proposition} \label{prop:theta-transform}
	Let $\gamma=\pmatrix abcd\in \SL_2(\Z)$. If $h\in L^*/L$ and $f$ satisfies the weight $\mu$ spherical property \eqref{eq:spherical-prop} then
	\begin{equation}
		\theta(\gamma z,f,h) = i^{\frac{q-p}2(\sgn c)} (cz+d)^\mu \sum_{k\in L^\ast/L} c(h,k)_\gamma \theta(z,f,k),
	\end{equation}
	where $c(h,k)_\gamma = \delta_{k,ah}e(\frac{ab}2\langle h,k \rangle)$ if $c=0$, and otherwise
	\begin{equation} \label{eq:c-h-k-formula}
		c(h,k)_\gamma = |\det Q|^{-\frac12}(\operatorname{vol} L)^{-1}|c|^{-\frac n2}\sum_{r\in L/cL} e\pfrac{a\langle h+r,h+r \rangle - 2\langle k, h+r \rangle + d\langle k,k \rangle}{2c}.
	\end{equation}
\end{proposition}

As in Theorem~1.9 of \cite{cipra}, we can construct functions satisfying the spherical property by taking combinations of products of Hermite polynomials.
For each integer $\mu\geq 0$ let $H_\mu(x)$ denote the Hermite polynomial
\begin{equation} \label{eq:hermite-def}
	H_\mu(x) = (-1)^\mu \exp(\tfrac 12x^2) \frac{d^\mu}{dx^\mu} \exp(-\tfrac 12x^2).
\end{equation}
Then $H_0(x) = 1$, $H_1(x) = x$, $H_2(x) = x^2-1$, etc.
These coincide with the Hermite polynomials $\He_\mu(x)$ in \cite[(18.7.12)]{DLMF}.
By \cite[(18.12.16)]{DLMF} we have the generating function
\begin{equation} \label{eq:hermite-gen}
	\sum_{\mu=0}^\infty \frac{H_{\mu}(x)}{\mu!}z^\mu = e^{xz-\frac 12z^2}.
\end{equation}

The theta kernel $\vartheta(z,w)$ which we use in Section~\ref{sec:lift} to define the Shimura lift is constructed by starting with a lattice of rank $3$ with associated bilinear form $\langle x,y \rangle$ that splits as 
\begin{equation} \label{eq:splitting}
	\langle x,y \rangle = \langle x_2,y_2 \rangle_1 + \langle (x_1,y_1),(x_3,y_3) \rangle_2,
\end{equation}
where $x=(x_1,x_2,x_3)$, $y=(y_1,y_2,y_3)$, and $\langle \cdot,\cdot \rangle_1$ and $\langle \cdot,\cdot \rangle_2$ are bilinear forms on $\R$ and $\R^2$, respectively.
A key property of $\vartheta(z,w)$ is that $\vartheta(z,iy)$ splits into a linear combination of products of the form $\vartheta_{1,\mu}(z)\vartheta_{2,\lambda-\mu}(z,y)$ where $\vartheta_{1,\mu}$ and $\vartheta_{2,\lambda-\mu}$ are theta series associated to $\langle \cdot,\cdot \rangle_1$ and $\langle \cdot,\cdot \rangle_2$.
In the next two subsections, we construct these theta functions, using notation consistent with the splitting \eqref{eq:splitting}.

\subsection{A theta series of rank 1}

We first construct a family of theta series that transform with multiplier system $\nu_N$.
The first element of this family is $\vartheta_{1,0}(z) = 2\eta(Nz)$.
The other elements  differ from this distinguished element only by a choice of Schwartz function $f$. 

Define
\begin{align}
	L_1 &= 12N \Z, \\
	L_1' &= N \Z, \\
	L_1^\ast &= \Z.
\end{align}
Then $L_1^\ast$ is dual to $L_1$ for the bilinear form $\langle x,y \rangle = \frac{xy}{12N}$ associated to $Q=\frac{1}{12N}$.
For $h\in L_1'$ let $h_2=h/N\in \Z$.

We will use the Schwartz function 
\begin{equation}
	f_{1,\mu}(x_2) = H_\mu\left(\sqrt{\tfrac{\pi}{3N}} \, x_2\right) e^{-\tfrac{\pi}{12N} x_2^2}.
\end{equation}
By Theorem~1.9 of \cite{cipra} the function $f_{1,\mu}$ has the spherical property \eqref{eq:spherical-prop} for weight $\mu+1/2$.
Thus it makes sense to form the theta series
\begin{equation} \label{eq:theta-1-mu-def}
	\vartheta_{1,\mu}(z) = \sum_{h\in L_1'/L_1} \chi_{12}(h_2) \theta(z,f_{1,\mu},h),
\end{equation}
where $\chi_{12} = \pfrac {12}{\bullet}$.
By \eqref{eq:theta-z-f-h-expanded}
we have
\begin{equation} \label{eq:theta-1-mu-fourier-exp}
	\vartheta_{1,\mu}(z) 
	= v^{-\frac\mu2} \sum_{x_2\in \Z} \chi_{12}(x_2) H_{\mu}\left(\sqrt{\tfrac13{\pi N v}} \, x_2\right) e\left(\tfrac {1}{24}Nx_2^2 z\right).
\end{equation}
In the special case $\mu=0$ we have
\begin{equation} \label{eq:theta-1-mu-eq-0-transform}
	\vartheta_{1,0}(z) = \sum_{x_2\in \Z} \chi_{12}(x_2) e\pfrac{Nx_2^2 z}{24} = 2\eta(Nz).
\end{equation}
We  use this to obtain a formula for the coefficients $c(h,k)_\gamma$ in the transformation law.

\begin{lemma} \label{lem:sum-c-h-k-nu}
	Let $h\in L_1'/L_1$ and $\gamma\in \Gamma_0(N)$.
	For $k\in L_1^\ast/L_1$, let $c_1(h,k)_\gamma:=c(h,k)_\gamma$ be as in Proposition~\ref{prop:theta-transform}.
	\begin{enumerate}
		\item If $k\notin L_1'/L_1$ then $c_1(h,k)_\gamma = 0$.
		\item If $k=Nk_2\in L_1'/L_1$ then
		\begin{equation} \label{eq:sum-c-h-k-nu}
			i^{-\frac 12(\sgn c)}\sum_{h\in L_1'/L_1} \chi_{12}(h_2) c_1(h,k)_\gamma = \chi_{12}(k_2) \nu_N(\gamma).
		\end{equation}
	\end{enumerate}
\end{lemma}

\begin{proof}
	If $c=0$ then (1) follows immediately from the formula given in Proposition~\ref{prop:theta-transform}.
	If $c\neq 0$, replace $r$ by $r+12c$ in \eqref{eq:c-h-k-formula}.
	Then, since $N\mid\gcd(c,r,h)$ we find that
	\begin{equation}
		c_1(h,k)_\gamma = e\pmfrac{-k}{N}c_1(h,k)_\gamma.
	\end{equation}
	Thus $c_1(h,k)_\gamma=0$ unless $N\mid k$, that is, $k\in L_1'/L_1$.

 By Proposition~\ref{prop:theta-transform}, \eqref{eq:theta-1-mu-def}, and \eqref{eq:theta-1-mu-eq-0-transform} we have
	\begin{equation}\label{eq:lin-comb-theta}
		\nu_N(\gamma) \sum_{h\in L_1'/L_1} \chi_{12}(h_2)\theta(z,f_{1,0},h) = i^{-\frac 12(\sgn c)} \sum_{k\in L_1'/L_1} \theta(z,f_{1,0},k) \sum_{h\in L_1'/L_1} \chi_{12}(h_2) c_1(h,k)_\gamma.
	\end{equation}
	For $h\in L_1'/L_1$, the Fourier expansion of $\theta(z,f_{1,0},h)$ is
	\begin{equation}
		\theta(z,f_{1,0},h) = \sum_{\ell\equiv h_2(12)} e\pfrac{N\ell^2z}{24},
	\end{equation}
	so $\theta(z,f_{1,0},h)$ and $\theta(z,f_{1,0},h')$ are linearly independent unless $h_2\equiv \pm h_2'\pmod{12}$.
    Equation \eqref{eq:sum-c-h-k-nu} now follows from \eqref{eq:lin-comb-theta}
    and the fact that $c_1(-h,-k)_\gamma = c_1(h,k)_\gamma$.
\end{proof}

The previous lemma gives a  transformation law for $\vartheta_{1,\mu}(z)$.

\begin{lemma} \label{lem:theta-1-transform}
	For $\mu\geq 0$ and for $\gamma=\pmatrix abcd\in \Gamma_0(N)$ we have
	\begin{equation}
		\vartheta_{1,\mu}(\gamma z) = \nu_N(\gamma)(cz+d)^{\mu+\frac12}\vartheta_{1,\mu}(z).
	\end{equation}
\end{lemma}

\begin{proof}
    By Proposition~\ref{prop:theta-transform} we have
	\begin{equation}
\vartheta_{1,\mu}(\gamma z) = i^{-\frac 12(\sgn c)} (cz+d)^{\mu+\frac 12} \sum_{k\in L_1'/L_1} \theta(z,f_{1,\mu},k) \sum_{h\in L_1^*/L_1} \chi_{12}(h_2) c_1(h,k)_\gamma.
	\end{equation}
 Lemma~\ref{lem:sum-c-h-k-nu} yields the desired result.
\end{proof}

\begin{remark}
    Since $\nu_N(-I)=-i$ we see
    that $\vartheta_{1,\mu} = 0$ whenever $\mu$ is odd.
\end{remark}

We can also determine the behavior of these theta functions under $z\mapsto -1/Nz$.

\begin{lemma} \label{lem:theta-1-fricke}
	For every $\mu\geq 0$ we have
	\begin{equation} \label{eq:theta-1-ep-fricke}
		\vartheta_{1,\mu}(-1/Nz) = i^{-\frac12}z^{\mu+\frac12} \vartheta_{1,\mu}(z/N).
	\end{equation}
\end{lemma}

\begin{proof}
    Write $f=f_{1,\mu}$ and $h=Nh_2$ for $h\in L_1'/L_1$.
    Then Proposition~\ref{prop:theta-transform} gives
	\begin{align}
	i^{\frac12}\vartheta_{1,\mu}(-1/z) 
		&= z^{\mu+\frac12} (12N)^{-\frac12} \sum_{h\in L_1'/L_1} \chi_{12}(h_2) \sum_{k\in L_1^\ast/L_1} e\pfrac{-h_2k}{12} \theta(z,f,k) \\
		&= z^{\mu+\frac12} (12N)^{-\frac12} \sum_{k(12N)} \theta(z,f,k) \sum_{h_2(12)} \chi_{12}(h_2) e\pfrac{-h_2k}{12}.
	\end{align}
	The inner Gauss sum evaluates to $\chi_{12}(k)\sqrt {12}$.
	Thus
	\begin{equation}
		i^{\frac12}\vartheta_{1,\mu}(-1/z) 
		= z^{\mu+\frac12}  N^{-\frac12} \sum_{k(12N)} \chi_{12}(k) \theta(z,f,k).
	\end{equation}
	Replacing $z$ by $Nz$ we obtain
	\begin{equation}
		i^{\frac12}\theta_{1,\mu}(-1/Nz) 
		= (Nv)^{-\frac\mu2} (Nz)^{\mu+\frac12} N^{-\frac12} \sum_{k(12N)} \chi_{12}(k) \sum_{x\equiv k(12N)} e\pfrac{ux^2}{24} f(\sqrt{Nv}\, x).
	\end{equation}
	Writing $k\equiv k_0 \pmod{12}$ with $k_0\in \Z/12\Z$ the latter equation becomes
	\begin{align}
		i^{\frac12}\theta_{1,\mu}(-1/Nz) 
		&= (Nv)^{-\frac\mu2} (Nz)^{\mu+\frac12} N^{-\frac12} \sum_{k_0(12)} \chi_{12}(k_0) \sum_{x\equiv k_0(12)} e\pfrac{ux^2}{24} f(\sqrt{Nv}\, x) \\
		&= N^{-\mu-\frac12} (Nz)^{\mu+\frac12} \sum_{k_0(12)} \chi_{12}(k_0) \theta(z/N,f,Nk_0) \\
		&= z^{\mu+\frac12} \sum_{h\in L_1'/L_1} \chi_{12}(h_2) \theta(z/N,f,k).
	\end{align}
	Equation \eqref{eq:theta-1-ep-fricke} follows.
\end{proof}

\subsection{A theta series of rank 2}
Next we construct a family of theta series that transform with (integral) weight $\lambda-\mu$, where $0\leq \mu \leq \lambda$.
We will eventually combine these with the theta series from the previous subsection to construct the two-variable theta kernel which  will be used in the Shimura lift.

Let
\begin{align}
	L_2 &= N\Z \oplus 6N\Z, \\
	L_2'&= \Z \oplus 6N\Z, \\
	L_2^\ast&= \Z \oplus 6\Z.
\end{align}
Then $L_2^\ast$ is dual to $L_2$ with respect to the bilinear form associated to $Q = \frac{1}{6N}\pmatrix{}{-1}{-1}{}$.
Let $\psi$ be a Dirichlet character modulo $N$ and for $x=(x_1,6Nx_3)\in L_2'$ define $\psi(x)=\psi(x_1)$.

\begin{lemma} \label{lem:theta-2-transform}
	Suppose that $f$ has the weight $\mu$ spherical property.
	Then for $\gamma=\pmatrix abcd\in \Gamma_0(N)$ we have
	\begin{equation}
		\sum_{h\in L_2'/L_2}\psi(h)\theta(\gamma z,f,h) = \psi(d)(cz+d)^\mu \sum_{h\in L_2'/L_2}\psi(h)\theta(z,f,h).
	\end{equation}
\end{lemma}

\begin{proof}
	Let $h\in L_2'/L_2$.
	By Proposition~\ref{prop:theta-transform} we have
	\begin{equation}
		\theta(\gamma z,f,h) = (cz+d)^\mu \sum_{k\in L_2^\ast/L_2} c(h,k)_\gamma \theta(z,f,k).
	\end{equation}
	If $c=0$ then $c(h,k)_\gamma=\delta_{k,ah}$ because $k=ah$ implies that $k=(*,0)$.
	If $c\neq 0$ then
	\begin{equation}
		c(h,k)_\gamma = N^{-1}|c|^{-1} \sum_{r\in L_2/cL_2} e\pfrac{a\langle h+r,h+r \rangle - 2\langle k,h+r \rangle + d\langle k,k \rangle}{2c}.
	\end{equation}
	We write $h=(h_1,0)$, $k=(k_1,6k_3)$, and $r=(Nr_1,6Nr_3)$.
	Then
	\begin{equation}
		c(h,k)_\gamma = N^{-1}|c|^{-1} e\pfrac{h_1k_3-dk_1k_3}{Nc} \sum_{r_1(c)} \sum_{r_3(c)} e\pfrac{-ah_1r_3-aNr_1r_3+k_1r_3+k_3r_1}{c}.
	\end{equation}
	The latter expression equals zero unless
	\begin{equation}
            k_1 = ah_1 \quad \text{ and } \quad k_3=0
            \iff k=ah.
	\end{equation}
        Writing $c=Nc'$, we obtain
        \begin{equation}
            c(h,k)_\gamma = \delta_{k,ah}N^{-1} |c|^{-1} \sum_{r_1, r_3(c)} e\pfrac{-ar_1r_3}{c'} = \delta_{k,ah}.
        \end{equation}
	Thus, for all $\gamma\in \Gamma_0(N)$ we have $c(h,k)_\gamma = \delta_{k,ah}$.
	It follows that
	\begin{align}
		\sum_{h\in L_2'/L_2}\psi(h)\theta(\gamma z,f,h)
		&= (cz+d)^\mu\sum_{h\in L_2'/L_2} \psi(h) \theta(z,f,ah)  \\
		&= \psi(d) (cz+d)^\mu\sum_{h\in L_2'/L_2} \psi(h) \theta(z,f,h).
	\end{align}
	This completes the proof.
\end{proof}

We specialize to $f=f_{2,\mu,y}$ where
\begin{equation} \label{eq:f-2-ep-y-def}
	f_{2,\mu,y}(x) =  H_{\mu}\left( \sqrt\frac{\pi}{3N} (y^{-1}x_1-yx_3) \right) \exp\pfrac{-\pi(y^{-2}x_1^2+y^2x_3^2)}{6N},
\end{equation}
and we define
\begin{equation}
	\vartheta_{2,\mu}(z,y) = v^{-\frac \mu 2}\sum_{h\in L_2'/L_2} \bar\psi(h_1) 
 \sum_{x\in L} [r(\sigma_z) f_{2,\mu,y}](x+h).
\end{equation}
By Theorem~1.9 of \cite{cipra} the function $f_{2,\mu,1}$ has the spherical property in weight $\mu$.

It will be useful in the next section to have the following expression for $\vartheta_{2,\mu}(z,y)$.

\begin{lemma} \label{lem:theta-2-poisson}
	For any $\mu\geq 0$ we have
	\begin{equation}
		\vartheta_{2,\mu}(z,y) = \frac{i^\mu v^{-\mu}}{\sqrt{6N}} y^{-1-\mu} \pfrac{\pi}{3N}^{\frac\mu2} \sum_{x_1,x_3\in \Z} \bar\psi(x_1) (x_1\bar z+x_3)^\mu \exp\left( \frac{-\pi}{6Nvy^2}|x_1 z+x_3|^2 \right).
	\end{equation}
\end{lemma}

\begin{proof}
	By \eqref{eq:weil-rep-def-integral} we have
	\begin{equation}
		\vartheta_{2,\mu}(z,y) = v^{\frac{1-\mu}2} \sum_{x_1,x_3\in \Z} \bar\psi(x_1) e(-ux_1x_3) f_{2,\mu,y}(\sqrt v \,x_1,6N\sqrt v \, x_3),
	\end{equation}
	where we have written $x\in L_2'$ as $x=(x_1,6Nx_3)$.
	Performing Poisson summation on $x_3$ we find that
	\begin{equation}
		\vartheta_{2,\mu}(z,y) = v^{\frac{1-\mu}2} \sum_{x_1,x_3\in \Z} \bar\psi(x_1) g(x_3),
	\end{equation}
	where
	\begin{equation}
		g(x_3) = \int_{-\infty}^\infty f_{2,\mu,y}(\sqrt v\, x_1, 6N\sqrt v\, t) e^{-2\pi i(ux_1+x_3)t} \, dt.
	\end{equation}
	Using \eqref{eq:f-2-ep-y-def} and making the change of variable $s=\sqrt{\pi v/3N}(y^{-1}x_1-6Nyt)$ we find that
 \begin{equation}
		g(x_3) = \frac{1}{2y\sqrt{3\pi N v}}  e\left(-\frac{x_1}{6N y^2}(x_1\bar z+x_3)\right) \int_{-\infty}^\infty H_\mu(s) e^{-\frac 12s^2+isw}\, ds,
	\end{equation}
	where
	\begin{equation}
		 w=\frac{\sqrt \pi}{y\sqrt{3N v}}(x_1 \bar z+x_3).
	\end{equation}
	By \cite[(18.17.23) and (18.17.27)]{DLMF} we have
	\begin{equation}
		\int_{-\infty}^\infty H_{\mu}(s) e^{-\frac 12s^2+isw} \, ds = i^\mu \sqrt{2\pi} w^{\mu} e^{-w^2/2},
	\end{equation}
 which holds for complex $w$ by analytic continuation.
	Thus
	\begin{equation}
		g(x_3) = \frac{i^\mu}{\sqrt{6N}} \pfrac{\pi}{3N}^{\frac\mu2}(vy^2)^{\frac{-1-\mu}2}(x_1\bar z+x_3)^\mu \exp\left( \frac{-\pi}{6Nvy^2}|x_1 z+x_3|^2 \right).
	\end{equation}
	The result follows.
\end{proof}

\subsection{A theta series of rank 3}
Combining the theta series of rank 1 with the theta series of rank 2 amounts to setting
\begin{align}
	L &= N\Z \oplus 12N\Z \oplus 6N\Z, \\
	L' &= \Z \oplus N\Z \oplus 6N\Z, \\
	L^* &= \Z \oplus \Z \oplus 6\Z.
\end{align}
Then $L^*$ is dual to $L$ with respect to the bilinear form
\begin{equation}
	\langle x,y \rangle = \frac{x_2y_2 - 2x_1y_3 - 2x_3y_1}{12N}
\end{equation}
of signature $(2,1)$ associated to $Q = \frac{1}{12N}\left(\begin{smallmatrix} &  & -2 \\  & 1 &  \\ -2 &  & \end{smallmatrix}\right)$.

For each $\lambda\geq 0$ we have the Hermite identity
\begin{equation}
	(x-iy)^\lambda = \sum_{\mu=0}^\lambda \binom{\lambda}{\mu}(-i)^\mu H_{\lambda-\mu}(x)H_{\mu}(y)
\end{equation}
which follows easily from the generating function identity \eqref{eq:hermite-gen}.
Let
\begin{align}
	f_{3}(x) &= (x_1-ix_2-x_3)^\lambda \exp\left(-\frac{\pi}{12N}(2x_1^2+x_2^2+2x_3^2)\right) \\
	&= \pfrac{\pi}{3N}^{-\frac\lambda2} \sum_{\mu=0}^\lambda \binom{\lambda}{\mu}(-i)^\mu f_{2,\lambda-\mu,1}(x_1,x_3)f_{1,\mu}(x_2).
\end{align}
It follows that $f_3$ has the spherical property in weight $\lambda+1/2$.
The next lemma is similar to the construction in Example~3 of \cite{niwa}.

\begin{lemma} \label{lem:theta3-transform-z}
Suppose that $f$ has the spherical property in weight $\lambda+1/2$, and define
\[
    \theta_{3}(z) = \sum_{h\in L'/L} \bar\psi(h_1) \chi_{12}(h_2) \theta(z,f,h).
\]
	Then for each $\gamma = \pmatrix abcd\in \Gamma_0(N)$ we have
	\begin{equation}
		\theta_3(\gamma z) = \nu_N(\gamma)\bar\psi(d)(cz+d)^{\lambda+\frac 12} \theta_3(z).
	\end{equation}
\end{lemma}

\begin{proof}
	By Proposition~\ref{prop:theta-transform} we have
	\begin{equation}
		\theta_3(\gamma z) = i^{-\frac 12(\sgn c)}  (cz+d)^{\lambda+\frac12} \sum_{k\in L^*/L} \theta(z,f,k) \sum_{h\in L'/L} \bar\psi(h_1)\chi_{12}(h_2) c(h,k)_\gamma.
	\end{equation}
	We employ the splitting \eqref{eq:splitting}, together with Lemma~\ref{lem:sum-c-h-k-nu} and the proof of Lemma \ref{lem:theta-2-transform} to evaluate $c(h,k)_\gamma$. 
	The only terms that are nonzero are those with $k\in L'/L$.
	For such $k$ we have
	\begin{align}
		i^{-\frac 12(\sgn c)} \sum_{h\in L'/L} &\bar\psi(h_1) \chi_{12}(h_2) c(h,k)_\gamma 
		\\
		&= \sum_{(h_1,0)\in L_2'/L_2} \delta_{k_1,ah_1} \bar\psi(h_1) \times i^{-\frac 12(\sgn c)}\sum_{Nh_2\in L_1'/L_1} \chi_{12}(h_2) c_1(h,k)_\gamma
		\\
		&= \bar\psi(dk_1)\nu_N(\gamma) \chi_{12}(k_2).
	\end{align}
	The lemma follows. 
\end{proof}

For $w=\xi+iy\in \H$ define
\begin{align}
	\vartheta^*(z,w) 
	= y^{-\lambda}\sum_{h\in L'/L} \bar\psi(h_1) \chi_{12}(h_2) \theta(z,\sigma_w f_{3},h),
\end{align}
where the action of $g\in \SL_2(\R)$ on functions is given by $gf(x) = f(g^{-1}x)$, and the action of $g$ on $x\in \R^3$ is given by
\begin{equation}\label{eq:SL2act}
	\pMatrix{x_1}{\frac 12 x_2}{\frac 12 x_2}{x_3} \mapsto g\pMatrix{x_1}{\frac 12 x_2}{\frac 12 x_2}{x_3}g^T.
\end{equation}
By \eqref{eq:weil-rep-def-integral} the Weil representation commutes with the action of $\operatorname{SO}(Q)$, that is, the group of matrices leaving $\langle \cdot, \cdot \rangle$ invariant.
Since the action \eqref{eq:SL2act} gives an isomorphism of $\operatorname{SO}(Q)$ with $\PSL_2(\R)$, the function $\sigma_w f_3$  has the spherical property.
Since 
\begin{equation}\label{eq:gxgx}
\langle gx, gx\rangle=\langle x,x\rangle
\end{equation}
we have
\begin{equation}
    \vartheta^*(z,w) = v^{\frac{1-\lambda}2}y^{-\lambda}\sum_{x\in L'}\bar\psi(x_1) \chi_{12}(x_2) e\left(\tfrac 12u\langle x,x \rangle\right) f_3\left(\sqrt v\, \sigma_w^{-1}x\right).\label{eq:thetastardef}
\end{equation}
It is straightforward to verify the relations
\begin{equation}\label{eq:sigma_gamma_w}
	\sigma_{\gamma w} = \gamma \sigma_w \kappa(\arg(cw+d)), \qquad \gamma = \pmatrix abcd \in \SL_2(\R),
\end{equation}
and
\begin{equation}
	f_3(\kappa(\phi)x) = e^{2i\lambda\phi}f_3(x).
\end{equation}
Thus
\begin{equation}\label{eq:f3spherical}
	f_3(\kappa(\arg(cw+d))^{-1}x) = \pfrac{c\bar w+d}{|cw+d|}^{2\lambda} f_3(x).
\end{equation}
Furthermore, for $\gamma\in \Gamma_0(6N)$ the map $x\mapsto \gamma x$ leaves the lattice $L'$ and the quantity $\chi_{12}(x_2)$ invariant and maps $\bar\psi(x_1)$ to $\psi^2(d)\bar\psi(x_1)$.
It follows that
\begin{equation} \label{eq:theta-star-transform-w}
	\vartheta^\ast(z,\gamma w) = \bar\psi^2(d) \pfrac{c\bar w+d}{|cw+d|}^{2\lambda}\im(\gamma w)^{-\lambda}y^{\lambda}\vartheta^\ast(z,w) = \psi^2(d) (c\bar w+d)^{2\lambda} \vartheta^\ast(z,w).
\end{equation}

As in  \cite{niwa} and \cite{cipra} $\vartheta^*(z,w)$ is not the correct theta kernel; instead we will use
\begin{equation}\label{eq:thetadef}
	\vartheta(z,w): = N^{-\frac\lambda2-\frac14}(-iz)^{-\lambda-\frac12}(6N)^{-\lambda}{\bar w}^{-2\lambda} \vartheta^*(-1/Nz,-1/6Nw).
\end{equation}
Then by Lemma~\ref{lem:theta3-transform-z}, Corollary~\ref{cor:eta_fricke} and equation \eqref{eq:theta-star-transform-w} we have
\begin{align} 
	\vartheta(\cdot,w)&\in \mathcal{M}_{\lambda+\frac12}\(N, \psi\nu\),  
 \label{eq:theta-bar-transform-z} 
 \\
	\overline{\vartheta(z, \cdot)}&\in \mathcal{M}_{2\lambda}\(6N, \psi^2\). 
 \label{eq:theta-bar-transform-w}
\end{align}
The following lemma provides a useful expression for $\vartheta(z,w)$ on the imaginary axis $w=iy$.
\begin{lemma} \label{lem:theta-z-iy-splitting}
	We have
	\begin{multline} \label{eq:theta-z-iy-splitting}
		\vartheta(z,iy) = 
		\sum_{\substack{\mu=0 \\ \mu \text{ even}}}^\lambda c_\mu  y^{1-\mu} \sum_{g\in \Z} \bar\psi(-g)g^{\lambda-\mu} 
		\\\times \sum_{\gamma\in \Gamma_\infty \backslash \Gamma_0(N)} \bar\psi(\gamma) (cz+d)^{-\lambda-\frac12} \exp\left(\frac{-6\pi y^2 g^2}{\im(\gamma z)}\right)\nu^{-1}(\gamma) \im(\gamma z)^{\mu-\lambda}  \vartheta_{1,\mu}\pfrac{\gamma z}{N},
	\end{multline}
	where
	\begin{equation}
		c_\mu = 
			\binom{\lambda}{\mu} 6^{\lambda-\mu+\frac12} \pfrac{3}{\pi}^{\frac\mu2} N^{\frac\lambda2-\frac\mu2+\frac14}.
	\end{equation}
\end{lemma}

\begin{remark}
    Recall that $\vartheta_{1,\mu}=0$ for odd $\mu$.
\end{remark}

\begin{proof}
	Since $\sigma_{iy}^{-1}x = (x_1/y, x_2, yx_3)$ we have
	\begin{equation}
		\vartheta^*(z,iy) = \pfrac{\pi}{3N}^{-\frac\lambda2} y^{-\lambda} \sum_{\substack{\mu=0 \\ \mu \text{ even}}} ^\lambda \binom{\lambda}{\mu}(-i)^\mu \vartheta_{1,\mu}(z)  \vartheta_{2,\lambda-\mu}(z,y).
	\end{equation}
	By Lemma~\ref{lem:theta-2-poisson} we have
	\begin{multline}
		\vartheta_{2,\lambda-\mu}(-1/Nz,y) = \pfrac{\pi}{3N}^{\frac{\lambda-\mu}2} \frac{i^{\lambda-\mu}v^{\mu-\lambda}}{y^{\lambda-\mu+1}\sqrt{6N}}z^{\lambda-\mu} 
		\\ \times \sum_{x_1,x_3\in \Z} \bar\psi(-x_1)(x_1+Nx_3\bar z)^{\lambda-\mu} \exp\left(\frac{-\pi|x_1+Nx_3z|^2}{6N^2vy^2}\right).
	\end{multline}
	Let $g=\sgn(x_3)\gcd(x_1,Nx_3)$ and write $Nx_3=gc$ and $x_1 = gd$. 
	Then the latter sum equals
	\begin{multline}
		\sum_{g\in \Z} \bar\psi(-g)g^{\lambda-\mu} \sum_{\substack{N\mid c\geq 0 \\ \gcd(c,d)=1}} \bar\psi(d) (c\bar z+d)^{\lambda-\mu} \exp\left(\frac{-\pi g^2|cz+d|^2}{6N^2vy^2}\right)
		\\= \sum_{g\in \Z} \bar\psi(-g)g^{\lambda-\mu} \sum_{\gamma\in \Gamma_\infty \backslash \Gamma_0(N)} \bar\psi(\gamma) (c\bar z+d)^{\lambda-\mu} \exp\left(\frac{-\pi g^2}{6N^2\im(\gamma z)y^2}\right),
	\end{multline}
	where $\gamma = \pmatrix **cd$.
	By Lemma~\ref{lem:theta-1-fricke} we have
	\begin{equation}
		\vartheta_{1,\mu}(-1/Nz) = i^{-\frac12}z^{\mu+\frac12}\vartheta_{1,\mu}(z/N).
	\end{equation}
	Since $\vartheta_{1,\mu}(z/N)$ transforms like $\eta(z)$ in weight $\mu+1/2$ we have
	\begin{equation}
		v^{\mu-\lambda}(c\bar z+d)^{\lambda-\mu} \vartheta_{1,\mu}(-1/Nz) = i^{-\frac12}z^{\mu+\frac12}\nu^{-1}(\gamma) \im(\gamma z)^{\mu-\lambda} (cz+d)^{-\lambda-\frac12} \vartheta_{1,\mu}\pfrac{\gamma z}{N}.
	\end{equation}
	It follows that
	\begin{multline}
		(-iz)^{-\lambda-\frac12}\vartheta^*(-1/Nz,iy) = \frac{(-1)^\lambda}{\sqrt {6N}} \sum_{\substack{\mu=0 \\ \mu \text{ even}}}^\lambda \binom{\lambda}{\mu} \pfrac{\pi}{3N}^{-\frac\mu2} y^{-2\lambda+\mu-1}\sum_{g\in \Z} \bar\psi(-g)g^{\lambda-\mu} 
		\\\times \sum_{\gamma\in \Gamma_\infty \backslash \Gamma_0(N)} \bar\psi(\gamma) (cz+d)^{-\lambda-\frac12} \exp\left(\frac{-\pi g^2}{6N^2\im(\gamma z)y^2}\right)\nu^{-1}(\gamma) \im(\gamma z)^{\mu-\lambda}  \vartheta_{1,\mu}\pfrac{\gamma z}{N}.
	\end{multline}
	From here it is straightforward to obtain \eqref{eq:theta-z-iy-splitting}.
\end{proof}

\subsection{Further properties of the theta kernel}
\label{sec:theta_up}

Here we assume that $(N, 6)=1$ and take
 \begin{equation*}
 H_{6N}:= \(\begin{matrix}0 & -1/\sqrt{6N} \\\sqrt{6N} & 0\end{matrix}\),
 \end{equation*}
 \begin{equation}\label{eq:wp6N}
W_p^{6N}= \pMatrix{\sqrt p\alpha}{1/\sqrt p}{6N\beta/\sqrt p}{\sqrt p}, \qquad  p\alpha-6N\beta/p=1,\qquad p\in \{2, 3\}.
\end{equation} 
Recalling the notation \eqref{eq:fslashstar}, the definition \eqref{eq:thetadef}
can be written in the form
\begin{equation}\label{eq:theta_thetastar}
\vartheta(z, w)=\vartheta^*(z, w)\sl  \mathcal{W}_{N, \lambda+\frac12}\sl^*_{2\lambda}H_{6N}
\end{equation}
where the first operator acts on $z$ and the second on $w$.
The following result describes the action of $U_p$ and $W_p^{6N}$ on these theta functions for $p\in\{2, 3\}$.  It will be important to 
determining the properties of the lifts at these primes.
 \begin{proposition}\label{prop:theta_eigen}
Suppose that $(N, 6)=1$.  For $p \in \{2,3\}$ the following are true (where all operators act on the variable $w$).
 \begin{align}
&\bar{\vartheta(z,w)}\sl U_p=p^{\lambda-1}\psi(p)\bar{\vartheta(z,w)},\\
&\bar{\vartheta(z,w)}\sl_{2\lambda} W_p^{6N}=-\psi(p)\bar{\vartheta(z,w)},\\
&\bar{\vartheta(z,w)}\sl_{2\lambda}H_{6N}\sl U_p=p^{\lambda-1}\bar\psi(p)\bar{\vartheta(z,w)}\sl_{2\lambda}H_{6N},\\
&\bar{\vartheta(z,w)}\sl_{2\lambda}H_{6N}\sl_{2\lambda} W_p^{6N}=-\bar\psi(p)\bar{\vartheta(z,w)}\sl_{2\lambda}H_{6N}.
\end{align}
 \end{proposition}

 \begin{proof}
  Since  $\mathcal{W}_{N, \lambda+1/2}$ acts on $z$, it  commutes with  the   operators in $w$.
By \eqref{eq:umdef} we have $\bar f\sl U_p=\bar{f\sl U_p}$ for any $f$, and $\sl^*_{2\lambda}H_{6N}$  is an involution. So  in order to prove  the proposition it will suffice by \eqref{eq:theta_thetastar} to prove the equivalent statements
 \begin{align}
&\vartheta^*(z,w)\sl^*_{2\lambda}H_{6N}\sl U_p=p^{\lambda-1}\bar\psi(p)\vartheta^*(z,w)\sl^*_{2\lambda}H_{6N},
\label{eq:theta*hup}\\
&\vartheta^*(z,w)\sl^*_{2\lambda}H_{6N}\sl^*_{2\lambda} W_p^{6N}=-\bar\psi(p)\vartheta^*(z,w)\sl^*_{2\lambda}H_{6N},
\label{eq:theta*hwp}\\
&\vartheta^*(z,w)\sl U_p=p^{\lambda-1}\psi(p)\vartheta^*(z,w),
\label{eq:theta*up}\\
&\vartheta^*(z,w)\sl^*_{2\lambda} W_p^{6N}=-\psi(p)\vartheta^*(z,w).
\label{eq:theta*wp}\\
\end{align}
We begin with a lemma.
\begin{lemma}\label{lem:Fricke}We have 
\begin{equation*}
	\vartheta^*(z,w)\sl^*_{2\lambda} H_{6N} 
	= y^{-\lambda}v^\frac{1-\lambda}2\psi(6)\sum_{x\in L'}\bar\psi\ptfrac{x_3}{N} \chi_{12}(x_2) e\(\tfrac 12u\langle x,x \rangle\) f_3\(\sqrt v\, \sigma_w^{-1}x\).
\end{equation*}
\end{lemma}
\begin{proof}
From \eqref{eq:SL2act} we have
 \begin{equation}\label{eq:H6nx}
 H_{6N}x=\(\tfrac{x_3}{6N},-x_2,6Nx_1\),
 \end{equation}
from which it follows that  $H_{6N}L'=L'$.
From \eqref{eq:thetastardef}, 
\eqref{eq:sigma_gamma_w} and  \eqref{eq:f3spherical},
we obtain
\begin{align*}
	\vartheta^*(z,w)\sl^*_{2\lambda} H_{6N} 
	&= (\bar{\sqrt{6N}w})^{-2\lambda}\(\im H_{6N}w\)^{-\lambda}v^\frac{1-\lambda}2\sum_{x\in L'}\bar\psi(x_1) \chi_{12}(x_2) e\(\tfrac 12u\langle x,x \rangle\) f_3\(\sqrt v\, \sigma_{H_{6N}w}^{-1}x\)\\
	&= y^{-\lambda}v^\frac{1-\lambda}2\sum_{x\in L'}\bar\psi(x_1) \chi_{12}(x_2) e\(\tfrac 12u\langle x,x \rangle\) f_3\(\sqrt v\,\sigma_w^{-1} H_{6N}^{-1}x\),
\end{align*}
and the lemma follows (using \eqref{eq:gxgx} and \eqref{eq:H6nx})  after replacing $x$ by $H_{6N}x$.
\end{proof}

Let $p\in \{2, 3\}$.
We first consider the statements involving $U_p$.
  Define
\begin{equation*}
\gamma_j:=\pMatrix{1/\sqrt p}{j/\sqrt p}0{\sqrt p}, 
\end{equation*}
so that
\begin{equation}\label{eq:theta*u3}
 \vartheta^*(z,w)\sl U_p=\mfrac1p\sum^{p-1}_{j=0}\vartheta^*(z,\gamma_jw).
 \end{equation}
For each $j$,  we have
\begin{equation*}
\sigma_{\gamma_j w}=\gamma_j\sigma_w,\ \ \  \operatorname{Im}(\gamma_jw)=\mfrac{y}p, \ \ \  \text{and}\ \ \  \langle \gamma_jx,\gamma_jx \rangle=\langle x,x \rangle.
\end{equation*}
We find that
\begin{equation}
        \gamma_jx=\(\tfrac{x_1+jx_2+j^2x_3}p,x_2+2jx_3,p x_3\),\qquad
  \gamma^{-1}_jx=\(p x_1-j x_2+\tfrac{j^2 x_3}p,x_2-\tfrac{2jx_3}p,\tfrac{x_3}p\), 
    \end{equation}
    \begin{equation}
        \gamma^{-1}_jL'=\left\{x \in \Z \oplus N\Z\oplus \mfrac{6N}p\Z: \ \ x_1+jx_2+j^2x_3 \equiv 0 \pmod p\right\}.
    \end{equation}

From these facts and \eqref{eq:thetastardef} we obtain
\begin{equation}\label{eq:theta_gj}
\begin{aligned}
\vartheta^*(z,\gamma_jw)
	&=\(\mfrac{y}p\)^{-\lambda}v^{\frac{1-\lambda}2}\sum_{x\in L'}\bar\psi\(x_1\) \chi_{12}(x_2) e\(\tfrac 12u\langle x,x \rangle\) f_3\(\sqrt v\, \sigma_w^{-1}\gamma^{-1}_jx\)\\
	&=p^\lambda\psi(p)y^{-\lambda}v^{\frac{1-\lambda}2} \sum_{x \in \gamma^{-1}_jL'} \bar\psi(x_1) \chi_{12}(x_2+2jx_3) e\(\tfrac 12u\langle x,x \rangle\) f_3\(\sqrt v\, \sigma_w^{-1}x\).
\end{aligned}
\end{equation}
Let $F_{z,w}(x) = y^{-\lambda}v^{\frac{1-\lambda}2} \bar\psi(x_1) e\(\tfrac 12u\langle x,x \rangle\) f_3(\sqrt v\, \sigma_w^{-1}x)$ for the moment; then
\begin{equation}
    \vartheta^*(z,w) \sl U_p = p^{\lambda-1} \psi(p) \sum_{\substack{x_1\in \Z \\ x_2 \in N\Z \\ x_3 \in (6N/p)\Z}} F_{z,w}(x) \sum_{\substack{j\bmod p \\ x_1+jx_2+j^2x_3 \equiv 0\spmod p}} \chi_{12}(x_2+2jx_3).
\end{equation}
The inner sum is periodic in $x_1$ modulo $p$, in $x_2$ modulo $12$, and in $x_3$ modulo $6$, so we can compute its value in every case.
We find that the inner sum equals zero unless $x_3\equiv 0\pmod p$, in which case it equals $\chi_{12}(x_2)$.
Thus we can change the condition $x_3\in (6N/p)\Z$ to $x_3\in 6N\Z$ at the cost of multiplying by $p$; it follows that
\begin{equation}
    \vartheta^*(z,w) \sl U_p = p^\lambda \psi(p) \vartheta^*(z,w).
\end{equation}

To establish \eqref{eq:theta*hup}, 
write (for the moment) $G(z, w)=\vartheta^*(z,w)\sl^*_{2\lambda} H_{6N}$.
Using Lemma~\ref{lem:Fricke},
 we find in 
 analogy with \eqref{eq:theta_gj} 
that
\begin{equation*}
	G(z, \gamma_jw)
	= p^\lambda\bar\psi(p) y^{-\lambda}v^\frac{1-\lambda}2\psi(6)\sum_{x\in \gamma_j^{-1}L'}\bar\psi\ptfrac{x_3}{N} \chi_{12}(x_2+2jx_3) e\(\tfrac 12u\langle x,x \rangle\) f_3\(\sqrt v\, \sigma_w^{-1}x\).
\end{equation*}
 The rest of the computation proceeds exactly as  above.

  For $x\in L'$  and $p\in \{2,3\}$ we find using \eqref{eq:wp6N} that
\begin{multline*}\label{eq:w36N}
W_p^{6N}x=(x_1', x_2', x_3')\\
=\(p\alpha^2 x_1+\alpha x_2+\tfrac{x_3}p,12N\alpha\beta x_1+\(p\alpha+\tfrac{6N\beta}p\)x_2+2x_3,\tfrac{36N^2\beta^2}p x_1+6N\beta x_2+px_3\).
\end{multline*}
We  have  $x'_2 \equiv  (2p\alpha-1)x_2\pmod{12}$ and (since $p\alpha\equiv 1\pmod N$) we have  $x'_1 \equiv \alpha x_1 \pmod N$.   
We also have $W_p^{6N}L'=L'$: the containment $W_p^{6N} L' \subseteq L'$ is immediate, and since $W_p^{6N}$ is an involution we get the other containment.
Arguing as in the proof of Lemma~\ref{lem:Fricke} gives
 \begin{equation}
	\vartheta^*(z,w)\sl_{2\lambda}^*W_p^{6N} 
	= y^{-\lambda}v^\frac{1-\lambda}2\sum_{x\in L'}\bar\psi(\alpha x_1) \chi_{12}\((2p\alpha-1)x_2\) e\(\tfrac 12u\langle x,x \rangle\) f_3\(\sqrt v\, \sigma_w^{-1}x\).
\end{equation}
From \eqref{eq:wp6N} we see that 
\begin{equation*}
2p\alpha-1\equiv \begin{cases} 7\pmod{12}\ \ \ \text{if $p=2$},\\
						5\pmod{12}\ \ \ \text{if $p=3$},\\
			\end{cases}
\end{equation*}
which  gives \eqref{eq:theta*wp}.

An analogous argument using Lemma~\ref{lem:Fricke} shows that
\begin{equation*}
\vartheta^*(z,w)\sl^*_{2\lambda} H_{6N} \sl^*_{2\lambda} W_p^{6N}=v^\frac{1-\lambda}2 y^{-\lambda}\psi(6)\sum_{x \in L'} \bar\psi\pmfrac{px_3}{N}
\chi_{12}\((2p\alpha-1)x_2\)e\(\tfrac12u\langle x,x \rangle\)f_3\(\sqrt{v}\sigma^{-1}_wx\),
\end{equation*}
which gives \eqref{eq:theta*hwp} and  finishes the proof of Proposition~\ref{prop:theta_eigen}.
\end{proof}

\section{\texorpdfstring{The Shimura lift when $(r,6)=1$}{liftr61}}
\label{sec:lift}
\subsection{Fourier expansion and transformation properties of the Shimura lift}
Here we prove a  version of the main theorem in which we do not require $(N,6)=1$.
Recall that for
	\begin{equation}\label{eq:Fdef}
		F(z) = \sum_{n\equiv r\spmod{24}} a(n) q^\frac n{24} \in  S_{\lambda+\frac12}(N,\psi\nu^r),
	\end{equation}
we have the lift
\begin{equation}
    \Sh_t(F) = \sum_{n=0}^\infty b(n) q^n,
\end{equation}
where $b(n)$ is defined as in \eqref{eq:stdef} by
	\begin{equation}
		\sum_{n=1}^\infty \frac{b(n)}{n^s} = L\(s-\lambda+1,\psi \ptfrac \bullet t\) \sum_{n=1}^\infty \frac{\chi_{12}(n)a(tn^2)}{n^s}.
	\end{equation}

\begin{theorem}\label{thm:transform}
Let $r$ be an integer with $(r,6)=1$ and let $t$ be a squarefree positive integer.
Suppose that $\lambda,N\in \Z^+$ and let $\psi$ be a Dirichlet character modulo $N$.  If $\lambda \geq 2$ then
$\Sh_t(F)\in S_{2\lambda}(6N, \psi^2)$, while if $\lambda = 1$ then $\Sh_t(F)\in M_{2\lambda}(6N, \psi^2)$ and $\Sh_t(F)$ vanishes at $\infty$. Furthermore,  the Hecke equivariance \eqref{equivariance3nmidr} holds.
\end{theorem}

We begin by assuming that $r=1$. Suppose that
\begin{equation}
	F(z) = \sum_{n\equiv 1\spmod{24}} a(n) q^{\frac n{24}} \in S_{\lambda+\frac12}(N,\psi\nu).
\end{equation}
We may  assume by \eqref{eq:krcond1} that $\psi(-1)=(-1)^\lambda$.
Define
\begin{equation}\label{eq:sh_lift_def}
	\Phi(w) = c^{-1}\int_{\Gamma_0(N)\backslash \H} v^{\lambda+\frac12} F(z) \overline{\vartheta(z,w)} \, \frac{dudv}{v^2},
\end{equation}
where $c:= 2(-12)^\lambda N^{\frac14 + \frac\lambda2}$. We will show that $\Phi=\Sh_1(F)$.

The proof of Proposition~2.8 of \cite{cipra}, with only cosmetic changes, shows that the integral defining $\Phi(w)$ converges absolutely.
The integral is well-defined because of \eqref{eq:theta-bar-transform-z}.
By \eqref{eq:theta-bar-transform-w} we have
\begin{equation}\label{eq:phi_trans}
	\Phi(w) \in \mathcal M_{2\lambda}(6N,\psi^2).
\end{equation}

Our next aim is to compute the Fourier expansion of $\Phi(w)$.
We first examine the behavior of $\Phi(iy)$ as $y\to\infty$.
By Lemma~\ref{lem:theta-z-iy-splitting}  we have
\begin{multline}
	\Phi(iy) = 2(-1)^\lambda c^{-1}\sum_{\substack{\mu=0 \\ \mu \text{ even}}}^\lambda c_\mu y^{1-\mu} \sum_{g=1}^\infty \psi(g) g^{\lambda-\mu} \\
	\times \sum_{\gamma\in \Gamma_\infty \backslash \Gamma_0(N)} \psi(\gamma)\nu(\gamma) \int_{\Gamma_0(N)\backslash \H} v^{\lambda+\frac12} (\overline{cz+d})^{-\lambda-\frac12} \overline{h(\gamma z,y)} F(z) \, \frac{dudv}{v^2},
\end{multline}
where
\begin{equation}
	h(z,y) = v^{\mu-\lambda}e^{-\frac{6\pi y^2 g^2}v}  \vartheta_{1,\mu}\pfrac{z}{N}.
\end{equation}
Since $\psi(\gamma)\nu(\gamma) F(z) = (cz+d)^{-\lambda-\frac12} F(\gamma z)$ we have
\begin{align}
	\Phi(iy) &= \frac{2(-1)^\lambda}{c}\sum_{\substack{\mu=0 \\ \mu \text{ even}}}^\lambda c_\mu y^{1-\mu} \sum_{g=1}^\infty \psi(g) g^{\lambda-\mu} \sum_{\gamma\in \Gamma_\infty \backslash \Gamma_0(N)} \int_{\Gamma_0(N)\backslash \H} \im(\gamma z)^{\lambda+\frac12} \overline{h(\gamma z,y)} F(\gamma z) \, \frac{dudv}{v^2} \\
		&= \frac{2(-1)^\lambda}{c}\sum_{\substack{\mu=0 \\ \mu \text{ even}}}^\lambda c_\mu y^{1-\mu} \sum_{g=1}^\infty \psi(g) g^{\lambda-\mu} \int_{\Gamma_\infty\backslash \H} v^{\lambda+\frac12} \overline{h(z,y)} F(z) \, \frac{dudv}{v^2}.
\end{align}
For fixed $v$ we have
\begin{equation} \label{eq:Phi-fourier-1}
	\int_0^1 \overline{h(u+iv,y)} F(u+iv) \, du = \sum_{n\equiv 1\spmod{24}} a(n) e^{-\frac{\pi n v}{12}} \int_0^1 \overline{h(u+iv,y)} e\pfrac{nu}{24} \, du.
\end{equation}
By \eqref{eq:theta-1-mu-fourier-exp}, when $\mu$ is even we have
\begin{equation}
	\vartheta_{1,\mu}\pfrac zN = 2N^\frac\mu2v^{-\frac\mu2} \sum_{n=1}^\infty \chi_{12}(n) H_{\mu}\left( \sqrt{\tfrac 13\pi n^2 v} \right) e\left(\tfrac{1}{24}n^2 z\right).
\end{equation}
Thus \eqref{eq:Phi-fourier-1} equals
\begin{equation}
	2 N^\frac\mu2 v^{\frac\mu2-\lambda} e^\frac{-6\pi y^2 g^2}v  \sum_{n=1}^\infty \chi_{12}(n) a(n^2) e^\frac{-\pi n^2v}6  H_{\mu}\left( \sqrt{\tfrac 13\pi n^2 v} \right)
\end{equation}
and we obtain
\begin{multline} \label{eq:phi-iy-mid}
	\Phi(iy) = \frac{4(-1)^\lambda}{c}\sum_{\substack{\mu=0 \\ \mu \text{ even}}}^\lambda c_\mu N^\frac\mu2 y^{1-\mu} \sum_{g=1}^\infty \psi(g) g^{\lambda-\mu} \sum_{n=1}^\infty \chi_{12}(n) a(n^2) \\
	\times \int_0^\infty v^\frac{\mu-3}2 e^{\frac{-6\pi y^2 g^2}v-\frac{\pi n^2v}6}  H_{\mu}\left( \sqrt{\tfrac 13\pi n^2 v} \right) \, dv.
\end{multline}
From this we can show that $\Phi(iy)$ decays polynomially, or better, as $y\to\infty$.

\begin{lemma}\label{lem:phi_inf}
As $y\to\infty$ we have $\Phi(iy)\ll_\lambda y^{-\lambda}$. In particular, $\Phi(i\infty)=0$.
\end{lemma}
\begin{proof}
	Since $F$ is a cusp form and $H_\mu(x)\ll (1+x)^\mu$, there exists a constant $\alpha>0$, depending only on $\lambda$, such that $a(n^2) H_{\mu}(\sqrt{\pi n^2 v/3}) \ll n^\alpha (1+v)^{\frac\mu2}$ for $\mu\leq \lambda$.
	Thus
	\begin{equation}
		\Phi(iy) \ll_{\lambda,N} \sum_{\mu=0}^\lambda y^{1-\mu} \int_0^\infty v^\frac{\mu-3}2(1+v)^{\frac\mu2} \sum_{g=1}^\infty g^{\lambda-\mu} e^\frac{-6\pi y^2 g^2}v \sum_{n=1}^\infty n^\alpha e^\frac{-\pi n^2 v}6 \, dv.
	\end{equation}
	For any $A,B>0$ we have
	\begin{equation}
		\sum_{m=1}^\infty m^A e^{-2B m^2} \leq e^{-B}\sum_{m=1}^\infty m^A e^{-B m^2} \ll e^{-B}\int_0^\infty x^A e^{-Bx^2}\, dx \ll_A B^{-\frac{A+1}2}e^{-B}.
	\end{equation}
	It follows that if $y\geq 1$ then
	\begin{align}
		\Phi(iy) 
		&\ll_{\lambda,N} \sum_{\mu=0}^\lambda y^{1-\mu} \int_0^\infty v^\frac{\mu-3}2(1+v)^{\frac \mu2} (y^2/v)^\frac{\mu-\lambda-1}2 e^{-3\frac{\pi y^2}v} v^{-\frac{\alpha+1}2} e^{-\frac{\pi v}{12}} \, dv \\
		&\ll_{\lambda,N} y^{-\lambda}\sum_{\mu=0}^\lambda \int_0^\infty v^{\frac{\lambda-\alpha-3}2}(1+v)^{\frac \mu2} e^{-\frac{3\pi}v -\frac{\pi v}{12}} \, dv \ll_{\lambda,N} y^{-\lambda}.
	\end{align}
	Thus $\Phi(i\infty)=0$.
\end{proof}

Now, following the argument  in the proof of Proposition 2.15 of \cite{cipra}, we can show that $\Phi(iy)$ decays exponentially as $y\to\infty$.

\begin{lemma} \label{lem:Phi-holomorphic}
There exist complex numbers $b(n)$ and $c(-n)$, $n\in \N$, such that
\begin{equation}
    \Phi(w) = \sum_{n>0} b(n) e(nw) + \sum_{n>0} c(-n) \Gamma(1-2\lambda,4\pi ny) e(-nw),
\end{equation}
where $\Gamma(a,z)$ is the incomplete gamma function (see \cite[\S 8.2]{DLMF}).
In particular, as $y\to \infty$ we have $\Phi(iy) \ll e^{-cy}$ for some $c>0$.
\end{lemma}

\begin{proof}
	As in Theorem~2.14 of \cite{cipra}, we have
	\begin{equation} \label{eq:PDE}
		\left(y\frac{\partial}{\partial w} - \lambda i\right)\frac{\partial}{\partial \bar w} \Phi = 0
	\end{equation}
	because $F$ is a cusp form.
	Since $\Phi(w+1)=\Phi(w)$, there is a Fourier expansion of the form
	\begin{equation}
		\Phi(w) = \sum_{n\in \Z} b(n,y)e(nw),
	\end{equation}
	for some coefficients $b(n,y)$.
	Equation \eqref{eq:PDE} implies that
	\begin{equation}
		b''(n,y) = (4\pi n - 2\lambda/y) b'(n,y).
	\end{equation}
	A basis of solutions for this differential equation is $\{1, \Gamma(1-2\lambda, -4\pi n y)\}$.
	Thus we have
	\begin{equation}
		\Phi(w) = \sum_{n\in \Z} b(n) e(nw) + \sum_{n\in \Z} c(n) \Gamma(1-2\lambda, -4\pi n y) e(nw)
	\end{equation}
	for some $b(n), c(n)\in \C$.
	Writing $\beta(n,y) = \left(b(n) + c(n) \Gamma(1-2\lambda, -4\pi n y)\right) e^{-2\pi n y}$, we see that
	\begin{equation}
		\int_0^1 |\Phi(w)|^2 \, d\xi = \sum_{m,n\in \Z} \beta(m,y)\bar \beta(n,y) \int_0^1 e((m-n)\xi)\, d\xi = \sum_{n\in \Z} |\beta(n,y)|^2.
	\end{equation}
	Since $\Phi(iy) \to 0$ as $y\to\infty$, the functions $\beta(n,y)$ must also have that property.
	By \cite[(8.11.2)]{DLMF}, we have
	\begin{equation} \label{eq:inc-gamma-growth}
		e^{-2\pi n y} \left| \Gamma(1-2\lambda, -4\pi n y) \right| \asymp |ny|^{-2\lambda} e^{2\pi n y}.
	\end{equation}
	Thus $c(n)=0$ for $n>0$ and $b(n)=0$ for $n<0$.
	We also have $b(0)+c(0)\Gamma(1-2\lambda)=0$ since $\Phi(i\infty)=0$, so
	\begin{equation}
		\Phi(iy) = \sum_{n > 0} b(n) e^{-2\pi n y} + \sum_{n>0} c(-n) \Gamma(1-2\lambda,4\pi ny) e^{2\pi n y}.
	\end{equation}
 This, together with \eqref{eq:inc-gamma-growth}, shows that $\Phi(iy)$ decays exponentially as $y\to\infty$.
\end{proof}

Let $\Lambda(s)$ denote the Mellin transform
\begin{equation}
	\Lambda(s) = \int_0^\infty y^{s} \Phi(iy) \, \frac{dy}{y}.
\end{equation}
By Lemma~\ref{lem:Phi-holomorphic}, the integral defining $\Lambda(s)$ is absolutely convergent for $\re(s)>1$.
Recall the expression \eqref{eq:phi-iy-mid} for $\Phi(iy)$.
We have
\begin{equation}
	2\sum_{g=1}^\infty \psi(g)g^{\lambda-\mu} \int_0^\infty y^{s-\mu} e^\frac{-6\pi y^2 g^2}{v} \, dy = \pfrac{v}{6\pi}^{\frac{s-\mu+1}2} \Gamma\left(\tfrac{s-\mu+1}2\right) L(s-\lambda+1,\psi),
\end{equation}
from which it follows that
\begin{multline}
	\Lambda(s) = \frac{2(-1)^\lambda}{c}L(s-\lambda+1,\psi) \sum_{\substack{\mu=0 \\ \mu \text{ even}}}^\lambda c_\mu N^{\frac\mu2} (6\pi)^\frac{\mu-s-1}{2} \Gamma\left(\tfrac{s-\mu+1}2\right) \sum_{n=1}^\infty \chi_{12}(n) a(n^2) \\ 
	\times \int_0^\infty v^{\frac s2} e^\frac{-\pi n^2 v}{6} H_{\mu}\left( \sqrt{\tfrac 13\pi n^2 v} \right) \, \frac{dv}{v}.
\end{multline}
By \eqref{eq:hermite-def} and a straightforward inductive argument, the latter integral equals
\begin{equation}
	2(-1)^\mu\pfrac{3}{\pi}^{\frac s2}n^{-s} \int_0^\infty t^{s-1} \frac{d^\mu}{dt^\mu} e^\frac{-t^2}{2}\, dt = 2^{-\frac\mu2}\pfrac{6}{\pi}^{\frac s2}n^{-s} (s-1)\cdots (s-\mu) \Gamma\left(\tfrac{s-\mu}2\right).
\end{equation}
Thus
\begin{multline}
	\Lambda(s) = \frac{2(-1)^\lambda}{c}\pfrac{6}{\pi}^{\frac s2} L(s-\lambda+1,\psi) \sum_{n=1}^\infty \frac{\chi_{12}(n)a(n^2)}{n^s} \\
	\times \sum_{\substack{\mu=0 \\ \mu \text{ even}}}^\lambda c_\mu N^{\frac\mu2} 2^{-\frac\mu2} (6\pi)^\frac{\mu-s-1}{2} (s-1)\cdots (s-\mu) \Gamma\left(\tfrac{s-\mu+1}2\right) \Gamma\left(\tfrac{s-\mu}2\right).
\end{multline}
We have
\begin{equation}
	(s-1)\cdots (s-\mu) \Gamma\left(\tfrac{s-\mu+1}2\right) \Gamma\left(\tfrac{s-\mu}2\right) = 2^{\mu+1-s}\sqrt{\pi}\,\Gamma(s),
\end{equation}
so, using that $\sum_{\mu=0}^{\lfloor\frac\lambda2\rfloor} \binom \lambda{2\mu} = 2^{\lambda-1}$, we conclude that
\begin{equation} \label{eq:mellin-transform}
	\Lambda(s) = (2\pi)^{-s} \Gamma(s) L(s-\lambda+1,\psi) \sum_{n=1}^\infty \frac{\chi_{12}(n)a(n^2)}{n^s}.
\end{equation}

Taking the inverse Mellin transform of \eqref{eq:mellin-transform}, we find that
	\begin{equation}
		\Phi(iy) = \sum_{n>0} \tilde b(n) e^{-2\pi n y}
	\end{equation}
	for some coefficients $\tilde b(n)$. 
By Lemma~\ref{lem:Phi-holomorphic} and the lemma on page 89 of \cite{cipra}, we see that $b(n) = \tilde b(n)$ and $c(n) = 0$ for all $n$.
Thus $\Phi$ is holomorphic on $\H$ and has a Fourier expansion of the form
\begin{equation}
	\Phi(z) = \sum_{n=1}^\infty b(n) q^n.
\end{equation}
It follows that
\begin{equation}
	\Lambda(s) = \sum_{n=1}^\infty b(n) \int_0^\infty y^s e^{-2\pi n y} \, \frac{dy}{y} = (2\pi)^{-s}\Gamma(s) \sum_{n=1}^\infty \frac{b(n)}{n^s}.
\end{equation}
Therefore we have the relationship
\begin{equation}
	\sum_{n=1}^\infty \frac{b(n)}{n^s} = L(s-\lambda+1,\psi) \sum_{n=1}^\infty \frac{\chi_{12}(n)a(n^2)}{n^s},
\end{equation}
that is, $\Phi= \Sh_1(F)$.

We are now ready to prove Theorem~\ref{thm:transform}.

\begin{proof}[Proof of Theorem~\ref{thm:transform}]
We first show that
\begin{equation} \label{eq:StF-transform}
    \mathcal S_t(F) \in \mathcal M_{2\lambda}(6N,\psi^2).
\end{equation}
When $r=t=1$, this is \eqref{eq:phi_trans}.
In the general case we may assume that $t\equiv r\pmod{24}$ (otherwise $\Sh_t(F)$ is identically zero).
	Then by Corollary~\ref{cor:v-op} we have $F\sl V_t\in S_{\lambda+1/2}\(Nt,\psi\ptfrac\bullet t \nu\)$.
We claim that
\begin{equation}\label{eq:s1_st}
\Sh_1(F\sl V_t) = \chi_{12}(t)\Sh_t(F)\sl V_t.
\end{equation}
Indeed, the Fourier coefficients $c(n)$ of $\Sh_1(F\sl V_t)$ are given by
\begin{align}
	\sum_{n=1}^\infty \frac{c(n)}{n^s} 
	&= L\(s-\lambda+1,\psi\ptfrac\bullet t\) \sum_{n=1}^\infty \frac{\chi_{12}(n)a(n^2/t)}{n^s} \\
	&= \frac{\chi_{12}(t)}{t^s}L\(s-\lambda+1,\psi\ptfrac\bullet t\) \sum_{n=1}^\infty \frac{\chi_{12}(n)a(tn^2)}{n^s}
 =\frac{\chi_{12}(t)}{t^s}\sum_{n=1}^\infty\frac{b(n)}{n^s},
\end{align}
where $\Sh_t(F)=\sum b(n)q^n$.
It follows that $c(n)=\chi_{12}(t)b(n/t)$ for all $n$, which proves \eqref{eq:s1_st}.
We have  $\Sh_t(F)\sl V_t \in \mathcal M_{2\lambda}(6Nt,\psi^2)$, so 
\eqref{eq:StF-transform} follows 
by  \cite[Lemma 4]{Li:1975}.  $\Sh_t(F)$ is holomorphic on $\H$ and vanishes at $\infty$ by \eqref{eq:s1_st}.

We next show that $\mathcal S_t(F)$ is a cusp form when $\lambda\geq 2$ or a modular form when $\lambda=1$.
Since $F$ is a cusp form we have $a(n)\ll n^{\frac \lambda 2+\frac 14}$,
so the Fourier coefficients $b(n)$ of $\mathcal S_t(F)$ satisfy
\begin{equation}
    b(n) = \sum_{jk=n}\psi(j)\pmfrac jtj^{\lambda-1}\chi_{12}(k)a(tk^2) \ll_t \sum_{jk=n} j^{\lambda-1}k^{\lambda+\frac 12} \ll_{t,\epsilon} n^{\lambda+\frac 12+\epsilon}
\end{equation}
for any $\epsilon>0$.
If $\lambda\geq 2$ then for sufficiently small $\epsilon$ we have $\lambda+\frac 12+\epsilon<2\lambda-1$, and a standard argument shows that $\mathcal S_t(F)$ vanishes at the cusps.
If $\lambda=1$ then a similar argument shows that $\mathcal S_t(F)$ is holomorphic at the cusps.

To finish the proof we need only to establish \eqref{equivariance3nmidr}, which is the subject of the next result.
\begin{proposition}\label{prop:hecke_equiv} Let $F$, $N$, $r$ and $t$ be as in the statement of Theorem~\ref{thm:transform}. For any prime $p\geq 5$ we have
\begin{equation}
\Sh_t(T_{p^2}F)=\pmfrac{12}pT_p\Sh_t(F).
\end{equation}
\end{proposition}
\begin{proof}
We may assume that $r\equiv t\pmod{24}$.
Writing $\Sh_t(F)=\sum b(n)q^n$, we have 
\begin{equation}\label{eq:bndef}
b(n)=\sum_{jk=n}\psi(j)\pmfrac jtj^{\lambda-1}\pmfrac{12}ka(tk^2).
\end{equation}
  Write $T_{p^2}F=\sum A(n)q^\frac n{24}$ as in \eqref{eq:heckedef24}.  Our goal is to show that 
\begin{equation}\label{eq:heckeshow}
b(pn)+\psi^2(p)p^{2\lambda-1}b\pmfrac np=\pmfrac{12}p\sum_{jk=n}\psi(j)\pmfrac jtj^{\lambda-1}\pmfrac{12}kA(tk^2).
\end{equation}
Write $n=p^\alpha n'$ with $p\nmid n'$, and for $\ell\geq0$  define
\begin{equation}\label{eq:sldef}
S_\ell=\psi(p)^{\alpha+1-\ell}\pmfrac pt^{\alpha+1-\ell}\pmfrac{12}p^\ell p^{(\alpha+1-\ell)(\lambda-1)}
\sum_{jk=n'}\psi(j)\pmfrac jtj^{\lambda-1}\pmfrac{12}ka(tp^{2\ell} k^2).
\end{equation}
A computation shows that the left side of 
\eqref{eq:heckeshow} is given by
\begin{equation}\label{eq:lefthecke}
\sum_{\ell=0}^{\alpha+1}S_\ell+p\sum_{\ell=0}^{\alpha-1}S_\ell.
\end{equation}
To compute the right side of \eqref{eq:heckeshow} we consider separately the three terms in the sum \eqref{eq:heckedef24} defining $A(tk^2)$.
After an involved computation we find that these three terms are given by
\begin{equation}\label{eq:righthecke}
\sum_{\ell=0}^{\alpha}S_{\ell+1}+S_0+p\sum_{\ell=1}^{\alpha}S_{\ell-1}.
\end{equation}
The proposition follows.
\end{proof}
This completes the proof of Theorem~\ref{thm:transform}.
\end{proof}

\subsection{\texorpdfstring{Image of the lift under $U_p$ and $W_p$.}{UPWP}}
For the rest of Section~\ref{sec:lift} we assume that $(N, 6)=1$.
For $p\in \{2,3\}$ recall the definition \eqref{eq:epdef}:
\begin{equation}
\ep_{p, r, \psi}=-\psi(p)\pmfrac{4p}r.
\end{equation}

\begin{proposition}\label{prop:shim_up}
Suppose that  $(r, 6)=(N, 6)=1$, that $t$ is a positive squarefree integer,  and that $\psi$ is a character modulo $N$.  Let $F\in S_{\lambda+1/2}(N,\psi\nu^r)$ and let 
$f:=\Sh_t(F)$.
For $p\in \{2, 3\}$ the following are true.
\begin{enumerate}
\item $f\sl U_p=-p^{\lambda-1}\ep_{p, r, \psi} f$.
\item $f\sl _{2\lambda} W_p^{6N}=\ep_{p, r, \psi} f$.
\item $f\|_{2\lambda}H_{6N}\sl U_p=-p^{\lambda-1}\bar\ep_{p, r, \psi} f\|_{2\lambda}H_{6N}$.
\item $f\|_{2\lambda}H_{6N}\sl _{2\lambda} W_p^{6N}=\bar\ep_{p, r, \psi} f\|_{2\lambda}H_{6N}$.
\end{enumerate}

\end{proposition}

\begin{proof}
In the case $r=1$ the proposition follows directly from \eqref{eq:sh_lift_def}
and Proposition~\ref{prop:theta_eigen}.
In the general case we may assume that    $t\equiv r\pmod{24}$ and therefore that $r=t$.  From Corollary~\ref{cor:v-op} we have
$F\|V_t\in S_{\lambda+1/2}\(Nt, \psi\ptfrac\bullet t\nu\)$.

From the first two assertions of the proposition with $r=1$  together with
 \eqref{eq:s1_st} we obtain
\begin{equation}
 \Sh_t(F)\|V_t\|U_p=-p^{\lambda-1}\ep_{p, 1, \psi\pfrac\bullet t} \Sh_t(F)\|V_t
 \end{equation}
and
 \begin{equation}
 \Sh_t(F)\|V_t\|_{2\lambda}W_p^{6Nt}=\ep_{p, 1, \psi\pfrac\bullet t} \Sh_t(F)\|V_t.
 \end{equation}
 Note that $V_t$ and $U_p$ commute, that $\|V_t\|_{2\lambda}W_p^{6Nt}=\sl_{2\lambda}W_p^{6N}\|V_t$, and that $h_1\|V_t=h_2\|V_t$ if and only if $h_1=h_2$.
It follows that 
\begin{equation}
 \Sh_t(F)\|U_p=-p^{\lambda-1}\ep_{p, 1, \psi\pfrac\bullet t} \Sh_t(F)
 \end{equation}
and
 \begin{equation}
 \Sh_t(F)\|_{2\lambda}W_p^{6N}=\ep_{p, 1, \psi\pfrac\bullet t} \Sh_t(F).
 \end{equation}
The first two assertions follow since  
$\ep_{p, r, \psi}=\ep_{p, 1, \psi\pfrac\bullet t}$.
 
We turn to the second two assertions.  From the $r=1$ case we have
 \begin{equation}
 \Sh_1(F\|V_t)\|_{2\lambda}H_{6Nt}\|U_p=-p^{\lambda-1}\bar\ep_{p, 1, \psi\pfrac\bullet t} \Sh_1(F\|V_t)\|_{2\lambda}H_{6Nt}
 \end{equation}
and
 \begin{equation}
 \Sh_1(F\|V_t)\|_{2\lambda}H_{6Nt}\|W_p^{6Nt}=\bar\ep_{p, 1, \psi\pfrac\bullet t} \Sh_1(F\|V_t)\|_{2\lambda}H_{6Nt}.
 \end{equation}
By  \eqref{eq:s1_st}
 we have 
\begin{equation} \Sh_1(F\|V_t)\|_{2\lambda}H_{6Nt}=\chi_{12}(t)\Sh_t(F)\|V_t\|_{2\lambda}H_{6Nt}
=\chi_{12}(t)\Sh_t(F)\|_{2\lambda}H_{6N}.
\end{equation}
The assertions follow from these facts.
\end{proof}
\subsection{Proof that the lift is a cusp form}  
After Theorem~\ref{thm:transform}, we need only to show that the lift of a form of weight $3/2$ which is orthogonal to all theta series is cuspidal.  It would be possible to modify the arguments of Cipra to prove this for general $N$.  Since these are quite involved, we choose instead to give an argument which leverages those results in the case when $(N, 6)=1$.  We make a general statement since the proof does not depend on the weight.

\begin{proposition}\label{prop:cuspform}
Suppose that $(r, 6)=(N, 6)=1$, that $t$ is a positive squarefree integer, and that  $\psi$ is a character modulo $N$.  Let $F\in S_{\lambda+1/2}(N,\psi\nu^r)$ where $\lambda\in \N$, and if $\lambda=1$ assume further that $F\in S_{3/2}^c(N,\psi\nu^r)$.
Then 
$\Sh_t(F)\in S_{2\lambda}(6N, \psi^2, \ep_{2,r, \psi},\ep_{3, r, \psi})$.
\end{proposition}

\begin{proof} Let $f=\Sh_t(F)$.
  After Theorem~\ref{thm:transform} and Proposition \ref{prop:shim_up} we need only to show that $f$ vanishes at all cusps (as mentioned, this has already been shown when $\lambda>1$).
To this end let  $\hat f=\Shl_t(F\sl V_{24})$.
By Proposition~\ref{prop:shcompare} we have the relationship
\begin{equation}\label{eq:ffhat}
f\sl \(1-U_2V_2-U_3V_3+ U_6V_6\)=\hat f\otimes\pmfrac{12}\bullet.
\end{equation}
By  Theorem 4.3 and Corollary 4.5 of \cite{cipra} we know that $\hat f\otimes\pfrac{12}\bullet$ is a cusp form.
We will use the following lemma.
\begin{lemma}\label{lem:cusps}
Suppose that $k, M\in \N$,  that $p$ is a prime  with $p\nmid M$, and that $\chi$ is a character modulo $M$.
Suppose that $f\in \mathcal{M}_k(pM, \chi)$ and that $f$ is holomorphic on $\H$.
  Suppose further that there exists $\ep_p$ with 
\begin{enumerate}
\item $f\sl _k W^{pM}_p=\ep_p f$,
\item $f\sl U_p=-\ep_p p^{\frac k2-1}f$, and
\item $f-f\sl U_p V_p\in S_k(p^2M, \chi)$.
\end{enumerate}
Then $f\in S_k(pM, \chi)$.
\end{lemma}
\begin{proof}[Proof of Lemma~\ref{lem:cusps}]  
Let the hypotheses be as in the statement.
Since $p\nmid M$,  it follows from Corollary~3.2 of \cite{kiral-young} that  each cusp of $\Gamma_0(pM)$  can be represented by a rational number of one of the forms
\begin{equation}\label{eq:cuspflavors}
\frac 1c \ \  \text{or} \ \   \frac1{pc} \ \ \ \text{where $p\nmid c$.}
\end{equation}
The lemma will follow from relating the expansions of $f$ and $f\sl U_pV_p$ at these cusps.
For each $c$ there is a positive integer $h_c$ for which we have an expansion of the form
\begin{equation}\label{eq:fat1c}
f\sl _k\pmatrix10c1=\sum_{n\in \Z}a(n)q_{h_c}^n, \qquad q_{h_c}:=e^\frac{2\pi i z}{h_c}.
\end{equation}

Suppose that $p\nmid c$.  By the assumptions and \eqref{eq:vmdef} we have
\begin{equation}\label{eq:cusp_pc2}
f\sl U_pV_p\sl _k\pmatrix10{pc}1
=-\ep_pp^{-1}f\sl _k\pmatrix p001\pmatrix10{pc}1
	=-\ep_pp^{-1}f\sl _k\pmatrix 10c1\pmatrix p001
	=-\ep_pp^{\frac k2-1}\sum a(n)q_{h_c}^{pn}.
\end{equation}
Writing $W^{pM}_p=\pMatrix{p\alpha }\delta{pM\beta}{p }$ as in 
\eqref{eq:atkinlehner}, 
we have
\begin{equation}\label{eq:cusp_pc1}
f\sl _k\pmatrix10{pc}1=\bar\ep_p f\sl _k  W_p^{pM}\pmatrix10{pc}1=\bar\ep_pf\sl _k\gamma\pmatrix p001,
\end{equation}
where 
\begin{equation}
\gamma=\pmatrix{\alpha+c\delta}\delta{cp+M\beta}{p}\in \SL_2(\Z).
\end{equation}

For any integer $j$ we can write
\begin{equation}\label{eq:gagp}
\gamma=\gamma'\pmatrix10{c}1\pmatrix1j01
\end{equation} 
 where 
 \begin{equation}
 \gamma'=\pMatrix**{M\beta(1+cj) + c^2j p}{p - c j p - j M \beta}.
 \end{equation}
 Choosing $j$ with  $j\equiv 0\pmod M$ and $cj\equiv -1\pmod p$ (which is possible since $p\nmid Mc$),
we can ensure that  $\gamma'\in \Gamma_0(pM)$.

Using \eqref{eq:cusp_pc1} and  \eqref{eq:gagp}, we compute 
(where $\zeta_{h_c}=e^\frac{2\pi i}{h_c}$)
\begin{equation}\label{eq:fatpc}
\begin{aligned}f\sl _k\pmatrix10{pc}1&=\bar\ep_pf\sl _k\gamma'\pmatrix10c1\pmatrix1j01\pmatrix p001\\
&=\bar\ep_p\chi(p)\sum a(n)\zeta_{h_c}^{nj}q_{h_c}^n\sl _k\pmatrix p001
=\bar\ep_p\chi(p)p^\frac k2\sum a(n)\zeta_{h_c}^{nj}q_{h_c}^{pn}.
\end{aligned}
\end{equation}
From \eqref{eq:cusp_pc2} and \eqref{eq:fatpc} we conclude that
\begin{equation}\label{eq:cusp_pc3}
\(f-f\sl U_pV_p\)\sl _k\pmatrix10{pc}1=\sum a(n) \left[\bar\ep_p\chi(p)\zeta_{h_c}^{nj}p^\frac k2+\ep_pp^{\frac k2-1}\right]q_{h_c}^{pn}.
\end{equation}
By assumption, $f-f\sl U_pV_p$ is a cusp form. Since the quantity in brackets is non-zero, it follows that $a(n)=0$ for all $n\leq 0$.
By \eqref{eq:fat1c} and \eqref{eq:fatpc} we conclude that $f$ vanishes both at  $1/c$ and at $1/pc$.  The lemma follows in view of \eqref{eq:cuspflavors}.
\end{proof}

Returning to the proof of Proposition~\ref{prop:cuspform},  consider the form $h:=f-f\sl U_2 V_2\in\mathcal{M}_{2\lambda}(12N, \psi^2)$.
We apply  Lemma~\ref{lem:cusps} to $h$ with $p=3$ and $M=4N$.
By \eqref{eq:ffhat} we have 
\begin{equation}
h-h\sl U_3V_3=f\sl \(1-U_2V_2-U_3V_3+ U_6V_6\)\in S_{2\lambda}(36N, \psi^2),
\end{equation}
so the third condition in Lemma~\ref{lem:cusps} is satisfied.
Since
\begin{equation}
\pmatrix2001W^{12N}_3=W^{6N}_3\pmatrix2001,
\end{equation}
 Proposition~\ref{prop:shim_up} gives
\begin{equation}
\begin{aligned}
h\sl _{2\lambda}W^{12N}_3&=f\sl _{2\lambda}W^{12N}_3+\tfrac12\ep_{2, r, \psi}f\sl_{2\lambda}\pmatrix2001\sl_{2\lambda}W_3^{12N}\\
&=f\sl _{2\lambda}W^{6N}_3+\tfrac12\ep_{2, r, \psi}f\sl _{2\lambda}W_3^{6N}\sl_{2\lambda}\pmatrix2001=\ep_{3, r, \psi} h.
\end{aligned}
\end{equation}
Since $U_3$ commutes with $U_2$ and $V_2$ we have
$h\sl U_3=-3^{\lambda-1}\ep_{3, r, \psi}h$.  
Lemma~\ref{lem:cusps} shows that $h\in S_{2\lambda}(12N, \psi^2)$.  A second application of the lemma with $p=2$ and $M=3N$ shows that $f\in S_{2\lambda}(6N, \psi^2)$ as desired.
\end{proof}

\subsection{Proof that the lifts are new}\label{subsec:newproof}
After Proposition~\ref{prop:cuspform} and Theorem~\ref{thm:transform} it suffices to prove that the lifts are  new at $2$ and $3$ in order to finish the proof of Theorem~\ref{thm:shimura-lift}.

Let $F$, $N$, $r$ and $t$ be as in the statement, and 
let $f=\Sh_t(F)\in S_{2\lambda}\(6N, \psi^2\)$.
From Lemma~\ref{lem:trace} we find for $p\in \{2,3\}$ that 
\begin{align*}\label{eq:trace_def}
\Tr^{6N}_{6N/p}f&= f+\bar\psi^2(p)p^{1-\lambda}f\sl _{2\lambda}W_p^{6N}\|U_p,\\
\Tr^{6N}_{6N/p}\(f\sl_{2\lambda}H_{6N}\)&= f\sl_{2\lambda}H_{6N}+\psi^2(p)p^{1-\lambda}f\sl_{2\lambda}H_{6N}\sl_{2\lambda}W_p^{6N}\|U_p.
\end{align*}
By Proposition~\ref{prop:shim_up} both of these expressions are zero, and we conclude by \eqref{eq:traceequiv}
that $\Sh_t(F)$ is new at $2$ and $3$.  This
 concludes the proof of Theorem~\ref{thm:shimura-lift}.

\section{The Shimura lift when \texorpdfstring{$(r,6)=3$}{(r,6)=3}}\label{sec:3midr}

In this section we sketch the proof of Theorem~\ref{thm:shimura-lift-r=3}.
Since the construction of the theta kernel and the Shimura lift in the case $(r,6)=3$ are  similar to the case $(r,6)=1$ we will omit most of the details.
As before we begin with  a proposition describing the transformation properties of the  lift.
Recall that for
	\begin{equation}\label{eq:F3def}
		F(z) = \sum_{n\equiv \frac r3\spmod{8}} a(n) q^\frac n{8} \in  S_{\lambda+\frac12}(N,\psi\nu^r),
	\end{equation}
the  lift is given by 
\begin{equation}
    \Sh_t(F) = \sum_{n=0}^\infty b(n) q^n,
\end{equation}
where the $b(n)$ are defined as in \eqref{eq:stdef3} by
	\begin{equation}
		\sum_{n=1}^\infty \frac{b(n)}{n^s} = L\(s-\lambda+1,\psi \ptfrac\bullet t\) \sum_{n=1}^\infty \frac{\chi_{-4}(n)a(tn^2)}{n^s}.
	\end{equation}
Here $\chi_{-4}=\pfrac{-4}{\bullet}$ and $\psi(-1)=\pfrac{-1}r(-1)^{\lambda}$ (recall \eqref{eq:krcond1}).

In the next two subsections we sketch the proof of the analogue of Theorem~\ref{thm:transform}.
\begin{theorem}\label{thm:transform3}
Let $r$ be an integer with $(r,6)=3$ and let $t$ be a squarefree positive integer.
Suppose that $\lambda,N\in \Z^+$ and let $\psi$ be a Dirichlet character modulo $N$.  If $\lambda \geq 2$ then
$\Sh_t(F)\in S_{2\lambda}(2N, \psi^2)$, while if $\lambda = 1$ then $\Sh_t(F)\in M_{2\lambda}(2N, \psi^2)$ and $\Sh_t(F)$ vanishes at $\infty$. Furthermore,  the Hecke equivariance \eqref{equivariance3midr} holds.
\end{theorem}

\subsection{The theta kernel}
We begin with the rank 3 lattices
\begin{align}
	L &= N\Z \oplus 4N\Z \oplus 2N\Z, \\
	L' &= \Z \oplus N\Z \oplus 2N\Z, \\
	L^* &= \Z \oplus \Z \oplus 2\Z.
\end{align}
Note that $L^*$ is dual to $L$ with respect to the bilinear form
\begin{equation}
	\langle x,y \rangle = \frac{x_2y_2 - 2x_1y_3 - 2x_3y_1}{4N}
\end{equation}
of signature $(2,1)$ associated to $Q = \frac{1}{4N}\left(\begin{smallmatrix} &  & -2 \\  & 1 &  \\ -2 &  & \end{smallmatrix}\right)$.
Let
\begin{align} 
	f_{3}(x) &= (x_1-ix_2-x_3)^\lambda \exp\left(-\frac{\pi}{4N}(2x_1^2+x_2^2+2x_3^2)\right) \\
	&= \pfrac{\pi}{N}^{-\frac\lambda2} \sum_{\mu=0}^\lambda \binom{\lambda}{\mu}(-i)^\mu f_{2,\lambda-\mu,1}(x_1,x_3)f_{1,\mu}(x_2), \label{eq:f3-decomp-r=3}
\end{align}
where
\begin{align}
	f_{1,\mu}(x_2) &= H_\mu\left(\sqrt{\tfrac{\pi}{N}} \, x_2\right) e^{-\tfrac{\pi}{4N} x_2^2}, \\
	f_{2,\mu,y}(x_1,x_3) &= H_{\mu}\left( \sqrt{\tfrac{\pi}{N}} (y^{-1}x_1-yx_3) \right) \exp\pfrac{-\pi(y^{-2}x_1^2+y^2x_3^2)}{2N}.
\end{align}
By slightly modifying the proofs of Lemmas~\ref{lem:sum-c-h-k-nu}, \ref{lem:theta-2-transform}, and \ref{lem:theta3-transform-z}, we obtain the following analogue of Lemma~\ref{lem:theta3-transform-z}.

\begin{lemma} \label{lem:theta3-transform-z-r=3}
Suppose that $f$ has the spherical property in weight $\lambda+1/2$.
For $h\in L'/L$, write $h=(h_1,Nh_2,2Nh_3)$ and define
\begin{equation}
	\theta_{3}(z) = \sum_{h\in L'/L} \bar\psi(h_1) \chi_{-4}(h_2) \theta(z,f,h).
\end{equation}
	Then for each $\gamma = \pmatrix abcd\in \Gamma_0(N)$ we have
	\begin{equation}
		\theta_3(\gamma z) = \nu_N^3(\gamma)\bar\psi(d)(cz+d)^{\lambda+\frac 12} \theta_3(z).
	\end{equation}
\end{lemma}

For $w=\xi+iy\in \H$ define
\begin{equation}
	\vartheta^*(z,w) 
	= y^{-\lambda}\sum_{h\in L'/L} \bar\psi(h_1) \chi_{-4}(h_2) \theta(z,\sigma_w f_{3},h).
\end{equation}
For $\gamma\in \Gamma_0(2N)$ the map $x\mapsto \gamma x$ leaves the lattice $L'$ and the quantity $\chi_{-4}(x_2)$ invariant and maps $\bar\psi(x_1)$ to $\psi^2(d)\bar\psi(x_1)$.
In analogy with \eqref{eq:theta-star-transform-w} we have
\begin{equation} \label{eq:theta-star-transform-w-r=3}
	\vartheta^\ast(z,\gamma w) = \psi^2(d) (c\bar w+d)^{2\lambda} \vartheta^\ast(z,w) \qquad \text{ for  }\gamma=\pmatrix abcd\in \Gamma_0(2N).
\end{equation}
The theta kernel is
\begin{equation}
	\vartheta(z,w) = N^{-\frac\lambda2-\frac14}(-iz)^{-\lambda-\frac12}(2N)^{-\lambda}{\bar w}^{-2\lambda} \vartheta^*(-1/Nz,-1/2Nw),
\end{equation}
which satisfies
\begin{align}
    \vartheta(\cdot,w) &\in \mathcal M_{\lambda+\frac 12} (N, \psi \nu^3), \\
    \overline{\vartheta(z,\cdot)} &\in \mathcal M_{2\lambda}(2N,\psi^2). \label{eq:theta-bar-transform-r=3}
\end{align}
In analogy with Lemma~\ref{lem:theta-z-iy-splitting}, $\vartheta(z,w)$ takes the following shape on the imaginary axis.

\begin{lemma} \label{lem:theta-z-iy-splitting-r=3}
	We have
	\begin{multline} \label{eq:theta-z-iy-splitting-r=3}
		\vartheta(z,iy) = 
		\sum_{\substack{\mu=0 \\ \mu \text{ odd}}}^\lambda c_\mu  y^{1-\mu} \sum_{g\in \Z} \bar\psi(-g)g^{\lambda-\mu} 
		\\\times \sum_{\gamma\in \Gamma_\infty \backslash \Gamma_0(N)} \bar\psi(\gamma) (cz+d)^{-\lambda-\frac12} \exp\left(\frac{-2\pi y^2 g^2}{\im(\gamma z)}\right)\nu^{-3}(\gamma) \im(\gamma z)^{\mu-\lambda}  \vartheta_{1,\mu}\pfrac{\gamma z}{N},
	\end{multline}
	where
	\begin{equation}
		\vartheta_{1,\mu}(z) = v^{-\frac\mu2} \sum_{x_2\in \Z} \chi_{-4}(x_2) H_\mu\left(\sqrt{\pi N v} \, x_2\right)e\left(\tfrac 18Nx^2z\right)
	\end{equation}
	and
	\begin{equation}
		c_\mu = 
			i\binom{\lambda}{\mu} 2^{\lambda-\mu+\frac12} \pi^{-\frac\mu2} N^{\frac\lambda2-\frac\mu2+\frac14}.
	\end{equation}
\end{lemma}

\begin{remark}
    The function $\vartheta_{1,\mu}$ is zero whenever $\mu$ is even.
\end{remark}

\begin{proof}
	We begin by using \eqref{eq:f3-decomp-r=3} to decompose $\vartheta^*(z,iy)$ as
	\begin{equation}
		\vartheta^*(z,iy) = \pfrac{\pi}{N}^{-\frac\lambda2} y^{-\lambda} \sum_{\substack{\mu=0 \\ \mu \text{ odd}}}^\lambda \binom{\lambda}{\mu}(-i)^\mu \vartheta_{1,\mu}(z)  \vartheta_{2,\lambda-\mu}(z,y),
	\end{equation}
	where $\vartheta_{2,\mu}(z,y)$ is defined similarly to its counterpart in Section~\ref{sec:theta-kernels}.
	By a computation similar to that in the proof of Lemma~\ref{lem:theta-2-poisson} we have
	\begin{align}
		\pfrac{N}{\pi}^\frac{\lambda-\mu}2 & \frac{y^{\lambda-\mu+1}\sqrt{2N}}{i^{\lambda-\mu}v^{\mu-\lambda}}z^{\mu-\lambda} 
		\vartheta_{2,\lambda-\mu}(-1/Nz,y) \\
		&=  \sum_{x_1,x_3\in \Z} \bar\psi(-x_1)(x_1+Nx_3\bar z)^{\lambda-\mu} \exp\left(\frac{-\pi|x_1+Nx_3z|^2}{2N^2vy^2}\right) \\
		&=  \sum_{g\in \Z} \bar\psi(-g)g^{\lambda-\mu} \sum_{\gamma\in \Gamma_\infty \backslash \Gamma_0(N)} \bar\psi(\gamma) (c\bar z+d)^{\lambda-\mu} \exp\left(\frac{-\pi g^2}{2N^2\im(\gamma z)y^2}\right),
	\end{align}
	where $\gamma = \pmatrix **cd$.
	As in Lemma~\ref{lem:theta-1-fricke} we have
	\begin{equation}
		\vartheta_{1,\mu}(-1/Nz) = i^{-\frac32}z^{\mu+\frac12}\vartheta_{1,\mu}(z/N).
	\end{equation}
	Since $\vartheta_{1,\mu}(z/N)$ transforms like $\eta^3(z)$ in weight $\mu+1/2$ we have
	\begin{equation}
		v^{\mu-\lambda}(c\bar z+d)^{\lambda-\mu} \vartheta_{1,\mu}(-1/Nz) = i^{-\frac32}z^{\mu+\frac12}\nu^{-3}(\gamma) \im(\gamma z)^{\mu-\lambda} (cz+d)^{-\lambda-\frac12} \vartheta_{1,\mu}\pfrac{\gamma z}{N}.
	\end{equation}
	It follows that
	\begin{multline}
		(-iz)^{-\lambda-\frac12}\vartheta^*(-1/Nz,iy) = \frac{(-1)^\lambda i}{\sqrt {2N}} \sum_{\substack{\mu=0 \\ \mu\text{ odd}}}^\lambda \binom{\lambda}{\mu} \pfrac{\pi}{N}^{-\frac\mu2} y^{-2\lambda+\mu-1}\sum_{g\in \Z} \bar\psi(-g)g^{\lambda-\mu} 
		\\\times \sum_{\gamma\in \Gamma_\infty \backslash \Gamma_0(N)} \bar\psi(\gamma) (cz+d)^{-\lambda-\frac12} \exp\left(\frac{-\pi g^2}{2N^2\im(\gamma z)y^2}\right)\nu^{-3}(\gamma) \im(\gamma z)^{\mu-\lambda}  \vartheta_{1,\mu}\pfrac{\gamma z}{N}.
	\end{multline}
	From here it is straightforward to obtain \eqref{eq:theta-z-iy-splitting-r=3}.
\end{proof}

\subsection{The Shimura lift}
We begin with the case when $r=3$ and $t=1$.
Suppose that 
\begin{equation}
	F(z) = \sum_{n\equiv 1\spmod 8} a(n) q^\frac n 8 \in S_{\lambda+\frac12}(N,\psi\nu^3),
\end{equation}
and define
\begin{equation}
	\Phi(w) = c^{-1}\int_{\Gamma_0(N)\backslash \H} v^{\lambda+\frac12} F(z) \overline{\vartheta(z,w)} \, \frac{dudv}{v^2},
\end{equation}
where $c = \tfrac 12i(-4)^{\lambda+1} N^{\frac14 + \frac\lambda2}$. 
As before, the integral defining $\Phi(w)$ converges absolutely and by \eqref{eq:theta-bar-transform-r=3} we have
\begin{equation}
	\Phi(w)\in \mathcal{M}_{2\lambda}\(2N, \psi^2\).
\end{equation}
Following the analogous computation in Section~\ref{sec:lift}, we find that
\begin{equation}
	\Phi(z) = \sum_{n=1}^\infty b(n) q^n,
\end{equation}
where
\begin{equation}
	\sum_{n=1}^\infty \frac{b(n)}{n^s} = L(s-\lambda+1,\psi) \sum_{n=1}^\infty \frac{\chi_{-4}(n)a(n^2)}{n^s}.
\end{equation}
The form $\Phi(z)$ is the $t=1$ Shimura lift $\Sh_1(F)$; this proves Theorem~\ref{thm:transform3} when $r=3$ and $t=1$.
For the remaining cases we use the fact that
\begin{equation}
    \mathcal S_1(F \sl V_t) = \chi_{-4}(t) \mathcal S_t(F)\sl V_t.
\end{equation}
The remainder of the proof of Theorem~\ref{thm:transform3} follows the proof of Theorem~\ref{thm:transform}.  In particular, the proof of Hecke equivariance uses  a direct analogue of Proposition~\ref{prop:hecke_equiv}.

\subsection{Proof of Theorem~\ref{thm:shimura-lift-r=3}}

The next result can be proved using the method of Proposition~\ref{prop:theta_eigen}. 
\begin{proposition}\label{prop:theta_eigen3}
Suppose that $N$ is odd. Then the following are true (where all operators act on the variable $w$).
 \begin{align}
&\bar{\vartheta(z,w)}\sl U_2=2^{\lambda-1}\psi(2)\bar{\vartheta(z,w)},\\
&\bar{\vartheta(z,w)}\sl_{2\lambda} W_2^{2N}=-\psi(2)\bar{\vartheta(z,w)},\\
&\bar{\vartheta(z,w)}\sl_{2\lambda}H_{2N}\sl U_2=2^{\lambda-1}\bar\psi(2)\bar{\vartheta(z,w)}\sl_{2\lambda}H_{2N},\\
&\bar{\vartheta(z,w)}\sl_{2\lambda}H_{2N}\sl_{2\lambda} W_2^{2N}=-\bar\psi(2)\bar{\vartheta(z,w)}\sl_{2\lambda}H_{2N}.
\end{align}
 \end{proposition}
 
The $r=3$ case of the next proposition follows from Proposition~\ref{prop:theta_eigen3}.  In the general case we argue as in  Proposition~\ref{prop:shim_up}, using the fact that 
\begin{equation*}
\ep_{2, 3, \psi\pfrac\bullet{r/3}}=\ep_{2, r, \psi}.
\end{equation*}
\begin{proposition}\label{prop:shim_up_3}
Suppose that $N$ is odd, that $(r, 6)=3$, that $t$ is a positive squarefree integer,  and that  $\psi$ is a character modulo $N$.  Let $F\in S_{\lambda+1/2}(N,\psi\nu^r)$ and let 
$f=\Sh_t(F)$.
Then we have the following:
\begin{enumerate}
\item $f\sl U_2=-2^{\lambda-1}\ep_{2, r, \psi} f$.
\item $f\sl _{2\lambda} W_2^{2N}=\ep_{2, r, \psi} f$.
\item $f\|_{2\lambda}H_{2N}\sl U_2=-2^{\lambda-1}\bar\ep_{2, r, \psi} f\|_{2\lambda}H_{2N}$.
\item $f\|_{2\lambda}H_{2N}\sl _{2\lambda} W_2^{2N}=\bar\ep_{2, r, \psi} f\|_{2\lambda}H_{2N}$.
\end{enumerate}
\end{proposition}

Let $F, N, r$, and $t$ be as in the statement of Theorem~\ref{thm:shimura-lift-r=3} and let $f=\Sh_t(F)$.
With $\hat f=\Shl_t(f\sl V_8)$,  Proposition~\ref{prop:shcompare} gives the relationship
\begin{equation}
    f\sl\(1-U_2V_2\)=\hat f\otimes\pmfrac{-4}\bullet.
\end{equation}
It follows from Theorem~\ref{thm:transform3}
and Cipra's work that 
 $f\sl\(1-U_2V_2\)\in S_{2\lambda}\(4N,\psi^2\)$, and using Lemma~\ref{lem:cusps} we conclude that $f\in S_{2\lambda}\(2N,\psi^2\)$.

 As in Section~\ref{subsec:newproof} we find that $f$ is new at $2$.  Together, these facts complete the proof of Theorem~\ref{thm:shimura-lift-r=3}.
\section{Quadratic congruences}
\label{sec:quadcong}
In this section, we prove Theorems~\ref{thm:cong1} and \ref{thm:cong2} using a generalization of the arguments of \cite{Ahlgren-Allen-Tang}.

\subsection{Background on modular Galois representations}
We  summarize some facts about modular Galois representations. See \cite{Hida} and \cite{Edixhoven} for more details. 
We begin with some notation. Let $k$ be an even integer and $N$ be a positive integer.
Throughout, let $\ell \geq 5$ be a prime such that $\ell \nmid N$.
Let $\bar{\Q} \subseteq \C$ be the algebraic closure of $\Q$ in $\C$. 
If $p$ is prime, then let $\bar{\Q}_{p}$ be a fixed algebraic closure of $\Q_{p}$ and fix an embedding $\iota_{p}:\bar{\Q} \hookrightarrow \bar{\Q}_{p}$. 
 The embedding $\iota_\ell $ allows us to view the coefficients of forms in $S_{k}(N)$ as elements of $\bar{\Q}_\ell $, and for each prime $p$,
the embedding $\iota_{p}$ allows us to view $G_{p}:=\Gal(\bar{\Q}_{p}/\Q_{p})$ as a subgroup of $G_{\Q}:=\Gal(\bar{\Q}/\Q)$.
For any finite extension $K/\Q$, let $G_{K}:=\Gal(\bar{K}/K)$.
 If $I_{p} \subseteq G_{p}$ is the inertia subgroup, then we denote the coset of absolute Frobenius elements above $p$ in $G_{p}/I_{p}$ by $\text{Frob}_{p}$. 

We denote by $\chi_\ell :G_{\Q}\rightarrow \Z^{\times}_\ell $ and $\omega_\ell :G_{\Q}\rightarrow \mathbb{F}^{\times}_\ell $ the $\ell$-adic and mod $\ell$ cyclotomic characters, respectively.
 We let $\omega_2,\omega'_2:I_\ell \rightarrow\mathbb{F}^{\times}_{\ell^2}$ denote Serre's fundamental characters of level $2$ (see \cite[~$\mathsection$$2.1$]{DukeSerre}).
 Both characters have order $\ell^2-1$, and we have 
 $\omega^{\ell+1}_2=\omega'^{\ell+1}_2=\omega_\ell $.

The following theorem is due to Deligne, Fontaine, Langlands, Ribet, and Shimura (see also \cite[Theorem $2.1$]{Ahlgren-Allen-Tang}). 
\begin{theorem}\label{bigGalThm}
Let $f=\sum a(n)q^{n} \in S_{k}(N)$ be a normalized Hecke eigenform. There is a continuous irreducible representation $\rho_{f}:G_{\Q} \rightarrow \GL_2(\bar{\Q}_\ell )$ with semisimple mod $\ell$ reduction $\bar{\rho}_{f}:G_{\Q} \rightarrow \GL_2(\bar{\mathbb{F}}_\ell )$ satisfying the following properties.
\begin{enumerate}
\item
If $p \nmid \ell N$, then $\rho_{f}$ is unramified at $p$ and the characteristic polynomial of $\rho_{f}(\Frob_{p})$ is $X^2-\iota_\ell (a(p))X+p^{k-1}$. 
\item
If $f \in S^{\new Q}_k(N)$, where $Q$ is a prime with $Q\mid\mid N$ then we have
\begin{equation}
\rho_{f}|_{G_{Q}} \cong \left(\begin{matrix}\chi_\ell  \psi & * \\0 & \psi \end{matrix}\right),
\end{equation}
where $\psi:G_{Q} \rightarrow \bar{\Q}^{\times}_\ell $ is the unramified character with $\psi({\Frob}_{Q})=\iota_\ell (a(Q))$.
\item
Assume that $2 \leq k \leq \ell+1$. 
\begin{itemize}
\item
If $\iota_\ell (a(\ell)) \in \bar{\Z}_\ell ^{\times}$, then $\rho_{f} |_{G_\ell }$ is reducible and we have
\begin{equation}
\rho_{f}|_{I_\ell } \cong \left(\begin{matrix}\chi_\ell ^{k-1} & * \\0 & 1 \end{matrix}\right).
\end{equation}
\item
If $\iota_\ell (a(\ell)) \not \in \bar{\Z}_\ell ^{\times}$, then $\bar{\rho}_{f}\sl_{G_\ell }$ is irreducible and $\bar{\rho}_{f}|_{I_\ell } \cong \omega^{k-1}_2 \oplus \omega'^{(k-1)}_2$.
\end{itemize}
\end{enumerate}
\end{theorem}

\begin{remark}
The Galois representations depend on the choice of embedding $\iota_\ell :\bar{\Q} \hookrightarrow \bar{\Q}_\ell $, but we have suppressed this from the notation.
\end{remark}

\subsection{Suitability}\label{subsec:suitable}
Recall the definition of suitability from the introduction.  Here we show that suitability holds for many spaces of forms.
\begin{proposition}\label{3.3 analogue}
Suppose that  $\ell \geq 5$ is prime and that $r$ is an odd integer.
Let $N$ be a squarefree, odd, positive integer with $\ell \nmid N$, and $3\nmid N$ if $3\nmid r$.
Let $\psi$ be a real Dirichlet character modulo $N$. Let $k$ be an even positive integer.
Then $(k,\ell)$ is suitable for every triple $(N,\psi,r)$ if the following conditions hold:
\begin{enumerate}
\item
$k \leq \ell-1$,
\item
$2^{k-1} \not \equiv 2^{\pm 1} \pmod\ell$,
\item
$k \neq \frac{\ell+1}2, \frac{\ell+3}2$,
\item
$\frac{\ell+1}{(\ell+1,k-1)}, \frac{\ell-1}{(\ell-1,k-1)} \geq 6$.
\end{enumerate}
 When $\ell > 5k-4$, we always have conditions $(1)$, $(3)$ and $(4)$.
\end{proposition}
\begin{proof}
The final assertion is easy to check.  Assume that $(r,6)=1$ and 
suppose that $f=\sum a(n)q^{n} \in S^{\new 2,3}_{k}(6N,\ep_{2,r,\psi},\ep_{3,r,\psi})$ is a normalized Hecke eigenform.
It follows from \cite[Theorem $2.47(b)$]{DDT} that there are four possibilities for the image of $\bar{\rho}_{f}$:
\begin{enumerate}
\item
$\bar{\rho}_{f}$ is reducible.
\item
$\bar{\rho}_{f}$ is dihedral, i.e. $\bar{\rho}_{f}$ is irreducible but $\bar{\rho}_{f} \sl_{G_{K}}$ is reducible for some quadratic $K/\Q$.
\item
$\bar{\rho}_{f}$ is exceptional, i.e. the projective image of $\bar{\rho}_{f}$ is conjugate to one of $A_{4}$, $S_{4}$, or $A_{5}$.
\item
The image of $\bar{\rho}_{f}$ contains a conjugate of $\SL_2(\mathbb{F}_\ell )$.
\end{enumerate}
We proceed by ruling out the first $3$ cases. 
By 
condition $(2)$ and \cite[Lemma $3.2$]{Ahlgren-Allen-Tang}, we see that $\bar{\rho}_{f}$ is irreducible. 
By condition $(3)$ and \cite[Lemma $3.2$]{Ahlgren-Allen-Tang}, we conclude that $\bar{\rho}_{f}$ is not dihedral.
 To rule out the exceptional case, it suffices to show that the projective image contains an element of order $\geq 6$.
 Suppose that $\iota_\ell (a(\ell)) \in \bar{\Z}^{\times}_\ell $.
  By Theorem~\ref{bigGalThm}, we know that
  \begin{equation}
  \rho_{f}\sl_{I_\ell } \cong \pMatrix{\chi_\ell ^{k-1}}{*}{0}{1}.
  \end{equation}
Since $\omega_\ell $ has order $\ell-1$, the projective image of $\bar{\rho}_{f}$ contains an element of order $\geq \frac{\ell-1}{(\ell-1,k-1)} \geq 6$.
If $\iota_\ell (a(\ell)) \not \in \bar{\Z}^{\times}_\ell $, then  Theorem~\ref{bigGalThm} implies that
  \begin{equation}
  \bar{\rho}_{f} \cong \pMatrix{\omega^{k-1}_2}{0}{0}{\omega'^{k-1}_2}.
  \end{equation}
  Since $\omega_2/\omega'_2$ has order $\ell+1$, we conclude by condition $(4)$ that the projective image of $\bar{\rho}_{f}$ contains an element of order $\frac{\ell+1}{(\ell+1,k-1)} \geq 6$. This completes the proof when $(r,6)=1$; the result when $(r,6)=3$ follows in a similar fashion.
\end{proof}

 \subsection{Preliminary results}
We begin by proving the main technical result used in the proof of Theorem~\ref{thm:cong1}. 
 \begin{theorem}\label{3.9 analogue}
 Suppose that  $\ell \geq 5$ is prime and that $r$ is an odd integer.
Let $N$ be a squarefree, odd, positive integer with $3\nmid N$ if $3\nmid r$.
Suppose that $\ell \nmid N$ and let  $\psi$ be a real Dirichlet character modulo $N$. Let $k$ be an even positive integer such that $(k,\ell)$ is suitable for $(N,\psi,r)$ and 
 let $m \geq 1$ be an integer.
  
  If $(r,6)=1$, then there exists a positive density set $S$ of primes such that if $p \in S$, then $p \equiv 1 \pmod{\ell^{m}}$ and $T_{p}f \equiv f \pmod{\ell^{m}}$ for each normalized Hecke eigenform $f \in S^{\new 2,3}_{k}(6N,\ep_{2,r,\psi},\ep_{3,r,\psi})$. If $(r,6)=3$, then we have the same result for $S^{\new 2}_{k}(2N,\ep_{2,r,\psi})$. 
 \end{theorem}
 \begin{proof} 
 We assume that $(r,6)=1$; the proof is similar when $(r,6)=3$.
 Choose a number field $E$ containing all  coefficients of all  normalized Hecke eigenforms in $S^{\new 2,3}_{k}(6N,\ep_{2,r,\psi},\ep_{3,r,\psi})$. 
 If $\lambda$ is the prime of $E$ induced by the embedding  $\iota_\ell$ then let $E_{\lambda}$ be the completion of $E$ at $\lambda$ with ring of integers $\mathcal{O}_{\lambda}$ and ramification index $e$.
 
 By Proposition~\ref{3.3 analogue} and \cite[Proposition $3.8$]{Ahlgren-Allen-Tang}, there exists $\sigma \in \Gal(\bar{\Q}/\Q(\zeta_\ell))$ such that $\bar{\rho}_f(\sigma)$ is conjugate to $\pmatrix{1}{1}{-1}{0}$ for every normalized Hecke eigenform $f \in S^{\new 2,3}_k(6N,\ep_{2,r,\psi},\ep_{3,r,\psi})$. 
 This implies that the characteristic polynomial of $\rho_f(\sigma)$ is congruent to $x^2-x+1 \pmod{\lambda}$. For a positive integer $w$, we argue as in the proof of \cite[Theorem $1.5$]{Ahlgren-Allen-Tang} to conclude that the characteristic polynomial of $\rho_f(\sigma^{\ell^{w-1}})$ is congruent to $x^2-x+1 \pmod{\lambda^{w}}$. 
We then apply the Chebotarev density theorem and Theorem~\ref{bigGalThm} as in the proof of \cite[Theorem $1.5$]{Ahlgren-Allen-Tang}. 
Our conclusion is
that for every positive integer $w$, there exists a positive density set $S_w$ of primes such that if $p \in S_w$, then $p \equiv 1 \pmod{\ell^w}$ and we have 
 \begin{equation}
 T_{p}f \equiv f \pmod{\lambda^w}
 \end{equation}
 for each normalized Hecke eigenform $f \in S^{\new 2,3}_k(6N,\ep_{2,r,\psi},\ep_{3,r,\psi})$.

We  claim that $S^{\new 2,3}_k(6N,\ep_{2,r,\psi},\ep_{3,r,\psi})$ has a basis $\{g_1, \dots, g_t\}$ consisting of forms whose coefficients are integers.
To see this, let $\{f_1,\dots,f_s\}$ be the set  of  normalized Hecke eigenforms in this space.
Then there is a basis composed of the $f_i$ and their images under various $V_d$ with $d\mid N$ 
(\cite[Lemma 2]{Li:1975} describes the interaction between the $V_d$ and the Atkin-Lehner operators).
The claim follows by a standard argument (see \cite[Corollary $6.5.6$]{Diamond-Shurman} or 
\cite[\S 6.1]{new2integerbasis}) once we know that this basis is stable under the action of the Galois group.  
But this is clear, since the Galois action commutes with  $V_d$ and since for $p\in \{2, 3\}$ the Atkin-Lehner eigenvalues at $p$ are determined by the $p$th Fourier coefficients of the $f_i$, which are integers by Corollary~\ref{cor:upeigen}.

Write
\begin{equation}
g_i=\sum^s_{j=1}\sum_{d \mid N} \alpha_{i,j,d}f_j\sl V_d
\end{equation}
and
\begin{equation}\label{eq:fjtogi}
f_j=\sum^t_{i=1} \beta_{i,j}g_i.
\end{equation}
 Enlarge $E$ to contain all of the coefficients $\alpha_{i,j,d}$ and $\beta_{i,j}$ and
  let $\pi \in \mathcal{O}_{\lambda}$ be a uniformizer. Choose $c_1\geq 0$ such that $\pi^{c_1}\alpha_{i,j,d} \in \mathcal{O}_{\lambda}$ for all   $\alpha_{i,j,d}$. Finally, for $M \in \Z^{+}$, let $w=eM+c_1$. Since the operators $T_p$ and $V_d$ commute when $d \mid N$ and $p \nmid N$, it follows for $p \in S_w$ with $p \nmid N$ that $p \equiv 1 \pmod{\ell^M}$ and
 \begin{equation}
 T_pg_i \equiv g_i \pmod{\lambda^{eM}} \ \ \ \ \text{ for  $i \in \{1,\dots,t\}$},
 \end{equation}
 which implies that
 \begin{equation}
 T_pg_i \equiv g_i \pmod{\ell^M} \ \ \ \ \text{ for  $i \in \{1,\dots,t\}$}.
 \end{equation}

 Let $\mathcal{O}_{\ell} \subset E$ be the subring of elements which are integral at all primes above $\ell$.  If we 
  set $M=m+c_2$, where $c_2 \in \Z^{+}$ is chosen so that $\ell^{c_2}\beta_{i,j} \in \mathcal{O}_{\ell}$ for all coefficients $\beta_{i,j}$, then \eqref{eq:fjtogi} shows that for 
  $p \in S_w$ we have $p \equiv 1 \pmod{\ell^m}$ and 
 \begin{equation}
T_pf_j \equiv f_j \pmod{\ell^m} \ \ \ \text{ for  $j \in \{1,\dots,s\}$}.
 \end{equation}
 The result follows.
 \end{proof}

In order to prove Theorem~\ref{thm:cong2}, we require the following analogue of \cite[Theorem $4.2$]{Ahlgren-Allen-Tang}.
 
 \begin{theorem}\label{Theorem 4.2 analogue}
 Let $k \in \Z^{+}$ be even.
  Suppose that $\ell \geq 5$ is prime and that there exists an integer $a$ for which $2^{a} \equiv -2 \pmod\ell $. Let $m \geq 1$ be an 
  integer. Let $N \in \Z^{+}$ be odd and squarefree with $\ell \nmid N$, and $3\nmid N$ if $3\nmid r$.
  If $(r,6)=1$, then there exists a positive density set $S$ of primes such that if $p \in S$, then $p \equiv -2 \pmod{\ell^{m}}$ and
  for each normalized Hecke eigenform $f= \sum a(n)q^{n} \in S^{\new 2,3}_{k}(6N, \ep_{2,r,\psi},\ep_{3,r,\psi})$, we have 
 \begin{equation} 
   T_{p}f \equiv -(-\ep_{2, r, \psi})^{a}p^{\frac k2-1}f \pmod{\ell^{m}}.
\end{equation}
If $(r,6)=3$, then the same result holds for $S^{\new 2}_{k}(2N, \ep_{2,r,\psi})$.
 \end{theorem}
 \begin{proof}
 We assume that $(r,6)=1$; the proof is similar if $(r,6)=3$.
 Let $E$ and $\lambda$ be defined as in the proof of Theorem~\ref{3.9 analogue}, and 
 let $w \in \Z^{+}$.
By \cite[Lemma $4.1$]{Ahlgren-Allen-Tang} we have
$2^{\ell^{w-1}(a-1)+1} \equiv -2 \pmod{\ell^{w}}.$ 
Since $\ep_{2, r, \psi} \in \{\pm 1\}$, and $a$ and $\ell^{w-1}(a-1)+1$ have the same parity, we may assume that $2^{a}\equiv -2 \pmod{\ell^{w}}$. 
It follows from Corollary~\ref{cor:upeigen} that $a(2)=-\ep_{2, r, \psi}2^{\frac k2-1}$.

We  apply part $(2)$ of Theorem~\ref{bigGalThm} and the Chebotarev density theorem as in the proof of \cite[~Theorem~$4.2$]{Ahlgren-Allen-Tang} to conclude that there is a positive density set $S_w$ of primes such that if $p \in S_w$, then
 \begin{equation}
 p= \chi_\ell(\Frob_{p}) \equiv \chi_\ell(\Frob^{a}_2) \equiv 2^{a} \equiv -2 \pmod{\ell^{w}}
 \end{equation}
 and that for all normalized Hecke eigenforms $f \in S^{\new 2,3}_{k}(6N, \ep_{2,r,\psi},\ep_{3,r,\psi})$, we have
 \begin{multline}
 a(p) =\Tr \rho_{f}(\Frob_{p}) \equiv \Tr \rho_{f}(\Frob^{a}_2)
   \equiv (-\ep_{2, r, \psi}2^{\frac k2-1})^a2^a+(-\ep_{2, r, \psi}2^{\frac k2-1})^a \\
   =(-\ep_{2, r, \psi})^{a}2^{a\(\frac k2-1\)}(2^{a}+1) \equiv -(-\ep_{2, r, \psi})^a p^{\frac k2-1} \pmod{\lambda^{w}}.
 \end{multline}
 Thus if $p\in S_w$ then for 
 all normalized Hecke eigenforms $f \in S^{\new 2,3}_{k}(6N, \ep_{2,r,\psi},\ep_{3,r,\psi})$, 
 we have
 \begin{equation}
 T_pf \equiv -(-\ep_{2, r, \psi})^{a}p^{\frac k2-1}f \pmod{\lambda^{w}}.
 \end{equation}
 
 We  then argue as in the end of the proof of Theorem~\ref{3.9 analogue} to see that for any $m \in \Z^{+}$, there exists $w \geq m$ such that if $p \in S_w$, then $p \equiv -2 \pmod{\ell^m}$ and
  \begin{equation}
 T_pf \equiv -(-\ep_{2, r, \psi})^{a}p^{\frac k2-1}f \pmod{\ell^{m}}
 \end{equation}
 for all normalized Hecke eigenforms $f \in S^{\new 2,3}_{k}(6N, \ep_{2,r,\psi},\ep_{3,r,\psi})$.
 \end{proof}

 \subsection{Proofs of Theorems~\ref{thm:cong1} and \ref{thm:cong2}}
 Let $\ell \geq 5$ be prime and $r$ be odd. Let $m \geq 1$ be an integer.
 Let $\lambda$ be a positive integer and $N \geq 1$ be an odd, squarefree integer  such that $\ell \nmid N$, and $3 \nmid N$ if $3 \nmid r$.
 Let $\psi$ be a real Dirichlet character modulo $N$.
 Recall that 
 $S_{\lambda+1/2}(N,\psi\nu^{r})_\ell  \subset S_{\lambda+1/2}(N,\psi\nu^{r})$
 consists of  forms with algebraic coefficients which are integral at all primes above $\ell$.
 Suppose that 
 \begin{equation}
 F(z)=\displaystyle \sum_{n \equiv r \spmod{24}} a(n)q^{\frac{n}{24}} \in S_{\lambda+\frac{1}2}(N, \psi\nu^{r})_\ell.
 \end{equation}
 Furthermore, if $\lambda=1$, suppose that $F \in S^c_{3/2}(N,\psi\nu^r)$. 
 
 For each squarefree $t$ we have
 $\mathcal{S}_{t}(F) \in S^{\new 2,3}_{2\lambda}(6N,\ep_{2,r,\psi},\ep_{3,r,\psi})$ if $(r,6)=1$ and  $\mathcal{S}_{t}(F) \in S^{\new 2}_{2\lambda}(2N,\ep_{2,r,\psi})$ if $(r,6)=3$.
 As $t$ ranges over all squarefree integers, there are only finitely many non-zero possibilities for $\mathcal{S}_t(F)$ modulo $\ell^m$; let $\{g_1,\dots,g_k\}$ be a set of representatives for these possibilities.  Let
  $\{f_1,\dots,f_s\}$ be the set of normalized Hecke eigenforms in $S^{\new 2,3}_{2\lambda}(6N,\ep_{2,r,\psi},\ep_{3,r,\psi})$ if $(r,6)=1$ and in $S^{\new 2}_{2\lambda}(2N,\ep_{2,r,\psi})$ if $(r,6)=3$.
 Write
 \begin{equation}\label{linearcombination}
 g_{j}=\sum^{s}_{i=1}\sum_{d \mid N}c_{i,j,d}f_{i} \sl V_d
 \end{equation}
 with $c_{i,j,d} \in \bar{\Q}$.
 Choose $c \geq 0$ such that $\ell^{c}c_{i,j,d}$ is integral at all primes above $\ell$ for all $c_{i,j,d}$.
 We require two short lemmas. 
Define
\begin{equation}\label{chir}
 \chi^{(r)}= \begin{cases} 
\(\frac{-4}{\bullet}\) & \text{if } 3 \mid r, \\
\(\frac{12}{\bullet}\) & \text{if }  3\nmid r.
\end{cases}
\end{equation} 
 
\begin{lemma}\label{5.1 analogue}
Suppose that $p, \ell \geq 5$ are prime, and that  $r$ is odd.
Let $N$ be a squarefree, odd, positive integer such that $p \nmid N$, $\ell \nmid N$, and $3 \nmid N$ if $3 \nmid r$.
Let $\psi$ be a real character modulo $N$ and
suppose that $F \in S_{\lambda+1/2}(N, \psi\nu^{r})_\ell$. 
Let $m \geq 1$ be an integer, and 
let the forms $f_{i}$ and the integer $c$ be defined as above. If  $\lambda_{p}$ is an integer such that $T_pf_i \equiv \lambda_{p}f_{i} \pmod{\ell^{m+c}}$ for all $i$, then
\begin{equation}
T_{p^2}F \equiv \chi^{(r)}(p)\lambda_{p}F \pmod{\ell^m} .
\end{equation}
\end{lemma}
\begin{proof}
Since $T_pf_i \equiv \lambda_pf_i \pmod{\ell^{m+c}}$ for all $i$, it follows from \eqref{linearcombination} and the fact that $T_p$ and $V_d$  commute for $d \mid N$ that 
\begin{equation}
T_pg_j \equiv \lambda_{p}g_j \pmod{\ell^m}\ \ \ \ \text{ for  $j \in \{1,\dots,k\}$} .
\end{equation}
Thus, for each squarefree $t$, it follows from \eqref{equivariance3nmidr} and \eqref{equivariance3midr} that
\begin{equation}
\mathcal{S}_{t}(T_{p^2}F)=\chi^{(r)}(p)T_p\mathcal{S}_{t}(F) \equiv \chi^{(r)}(p)\lambda_{p}\mathcal{S}_{t}(F) \pmod{\ell^m} .
\end{equation}
A standard argument shows that
\begin{equation}
F \equiv 0 \pmod{\ell^{m}} \iff \Sh_t(F)\equiv 0 \pmod{\ell^{m}} \ \ \ \text{ for all squarefree $t$.}
\end{equation}
The result follows.
\end{proof}
The next lemma explains how to produce congruences from Lemma~\ref{5.1 analogue}.
\begin{lemma}\label{5.2 analogue}
Suppose that  $p, \ell \geq 5$ are prime and that $r$ is odd.
Let $N$ be a squarefree, odd, positive integer such that $p \nmid N$, $\ell \nmid N$, and $3 \nmid N$ if $3 \nmid r$.
Let $m \geq 1$ be an integer.
Let  $\psi$ be a real character modulo $N$ and suppose that 
\begin{equation}\label{exp24thpowers}
F=\displaystyle\sum_{n \equiv r \spmod{24}} a(n)q^{\frac{n}{24}} \in S_{\lambda+\frac{1}2}(N,\psi\nu^{r})_\ell.
\end{equation} 
Suppose that  there exists $\alpha_{p} \in \{\pm 1\}$ with
\begin{equation}
T_{p^2}F \equiv \alpha_{p}p^{\lambda-1}F \pmod{\ell^m} .
\end{equation}
Then we have
\begin{equation}
a(p^2n) \equiv 0 \pmod{\ell^m}  \ \ \ \text{ if } \ \ \ \pmfrac{n}{p} =\alpha_{p} \pmfrac{12}p \pmfrac{-1}p^{\frac{r-1}{2}}\psi(p).
\end{equation}
\end{lemma}

\begin{proof}
This follows from the definition of the Hecke operator in \eqref{eq:heckedef24}, which is 
valid for $p\geq 5$ in all cases by the remark following \eqref{eq:heckedef8}.
 If we write $T_{p^2}F=\sum c(n)q^{\frac{n}{24}}$ and $n$ satisfies the quadratic condition above, then the third term defining $c(n)$ does not contribute and the middle term  becomes 
 $\alpha_{p}p^{\lambda-1}a(n)$.
\end{proof}
We now prove Theorem~\ref{thm:cong1} and Theorem~\ref{thm:cong2}.
\begin{proof}[Proof of Theorem~\ref{thm:cong1}]
Let $c$ be the integer defined in Lemma~\ref{5.1 analogue}.
If $(r,6)=1$, we apply Theorem~\ref{3.9 analogue} to conclude that there exists a positive density set $S$ of primes such that if $p \in S$, then we have $p \equiv 1 \pmod{\ell^m} $, $p \nmid 6N$ and
\begin{equation}
T_pf  \equiv f \pmod{\ell^{m+c}}
\end{equation}
 for each normalized Hecke eigenform $f \in S^{\new 2,3}_{2\lambda}(6N,\ep_{2,r,\psi},\ep_{3,r,\psi})$.
If $(r, 6)=3$ we obtain the same conclusion for 
 $S^{\new 2}_{2\lambda}\(2N,\ep_{2,r,\psi}\)$.
For such $p$, Lemma~\ref{5.1 analogue} 
implies that
\begin{equation}
T_{p^2}F \equiv \chi^{(r)}(p)F \pmod{\ell^{m}} .
\end{equation}
The result follows from Lemma~\ref{5.2 analogue}.
\end{proof}

 \begin{proof}[Proof of Theorem~\ref{thm:cong2}]
 Suppose that $\ell \geq 5$ is a prime such that $2^{a} \equiv -2 \pmod{\ell} $ for some integer $a$. 
 If  $(r,6)=1$, then
 by Theorem~\ref{Theorem 4.2 analogue}, there exist  $\beta \in \{\pm 1\}$ and a positive density set $S$ of primes such that if $p \in S$, then $p \equiv -2 \pmod{\ell^{m}}$, $p \nmid 6N$, and for each normalized Hecke eigenform $f \in S^{\new 2,3}_{2\lambda}(6N,\ep_{2,r,\psi},\ep_{3,r,\psi})$, we have
 \begin{equation}
 T_pf \equiv \beta p^{\lambda-1}f \pmod{\ell^{m+c}}.
 \end{equation}
 If $(r, 6)=3$ the same holds for $S^{\new 2}_{2\lambda}\(2N,\ep_{2,r,\psi}\)$.
 
 In either case, for  such $p$ Lemma~\ref{5.1 analogue} implies that
 \begin{equation}
T_{p^2}F \equiv \beta\chi^{(r)}(p)p^{\lambda-1}F \pmod{\ell^m} .
\end{equation}
The result follows from Lemma~\ref{5.2 analogue}. 
 \end{proof}

 \begin{remark}
 We could also prove a theorem similar to \cite[Theorem $5.1$]{RJD}. 
 For a fixed $\alpha \in \Z$, we can assume without loss of generality that $\mathcal{O}_{\lambda}$ has the property that
 \begin{equation}
 \text{ the polynomial $x^2-\alpha x+1$ factors in $\mathcal{O}_{\lambda}$ with roots $\alpha_{1}$ and $\alpha_2$.}
 \end{equation}
 If $(k,\ell)$ is suitable for $(N,\psi,r)$, $\alpha \not \equiv \pm 2 \pmod\ell $ and $(r,6)=1$, then one could show that there exists a positive density set $S$ of primes such that if $p \in S$, then $p \equiv 1 \pmod{\ell^{m}}$ and for each normalized Hecke eigenform $f \in S^{\new 2,3}_{k}(6N,\ep_{2,r,\psi},\ep_{3,r,\psi})$, we have
 \begin{equation}
 T_{p}f  \equiv (\alpha_{1}^{\ell^{m-1}}+\alpha_2^{\ell^{m-1}})f \pmod{\lambda^{m}}.
 \end{equation}
 A similar result holds when $(r,6)=3$.

With this tool in hand, we can prove congruences similar to those given in \cite[Theorem $1.2$]{RJD}.
In particular, if $\alpha$ is an integer which satisfies $\alpha \not \equiv -2 \pmod\ell $, then one could show that there is a positive density set $S$ of primes such that if $p \in S$, then $p \equiv 1 \pmod{\ell^m} $ and
\begin{equation}
a(p^2n) \equiv (\alpha-1)\chi^{(r)}(p)a(n) \pmod{\pi^m}  \ \ \ \ \text{ if } \ \ \ \(\frac{n}{p}\)=\begin{cases} 
\(\frac{-1}{p}\)^{\frac{r-1}{2}}\psi(p) & \text{if } 3 \nmid r, \\
\(\frac{-3}{p}\)\(\frac{-1}{p}\)^{\frac{r-1}{2}}\psi(p) & \text{if }  3\mid r,
\end{cases}
\end{equation}
where $\pi$ is a prime above $\ell$ in a large enough number field.
 \end{remark}

\section{Congruences for colored generalized Frobenius partitions}\label{sec:GFP}
Here we give an extended example which illustrates the use of our main results to prove congruences for the colored Frobenius partitions described in the Introduction.  A complete treatment will be the subject of a future paper.

  As described in the Introduction, a result of  Chan, Wang and Yang \cite[Theorem $2.1$]{Chan-Wang-Yang}  shows that if $m$ is a positive odd integer, then 
\begin{equation}\label{eq:cwyfrob}
 A_m(z):=\prod_{n \geq 1}(1-q^n)^m\sum^\infty_{n=0} c\phi_m(n)q^n \in M_\frac{m-1}2\(m, \ptfrac\bullet m\).
\end{equation}
Here we discuss the case $m=5$.
Letting $\Delta$ denote the normalized cusp form of weight $12$ on $\SL_2(\Z)$, 
  we define (for primes $\ell\geq 7$)
\begin{equation}\label{eq:c5gen}
F_\ell = \eta^{-5\ell}\, T_\ell\(\Delta^\frac{5(\ell^2-1)}{24}A_5\)\in M_\frac{5\ell^2-5\ell-1}2^!\(5,\ptfrac\bullet 5\nu^{-5\ell}\).
\end{equation}
Then 
\begin{equation}\label{eq:flcong}
F_\ell \equiv  \sum c\phi_5\pfrac{\ell n+5}{24}q^{\frac n {24}} \pmod\ell  
\end{equation}
(with additional work it can  be shown that      $F_\ell$ is congruent modulo $\ell$ to an element of $S_{(\ell-2)/2}\(5,\ptfrac\bullet 5\nu^{-5\ell}\)$).  Here we will discuss only the primes $\ell=7, 11$, and $13$.

When $\ell=7, 11$, we find  that $F_\ell\equiv 0\pmod\ell$.  In other words we have the congruences
\begin{align} c\phi_5(7n+4)&\equiv 0\pmod 7,\\
                    c\phi_5(11n+8)&\equiv 0\pmod {11}.		
\end{align}
We note that there are similar congruences for $c\phi_7$ and $c\phi_{11}$, which can also be deduced from (1.13)--(1.15) of \cite{Chan-Wang-Yang}. These are analogues of Ramanujan's well known congruences for $p(n)$.

When $\ell=13$ the situation is  more interesting.
Define
\begin{equation}\label{eq:Ftildef}
\tilde F_{13}(z):=6\frac{\eta^{12}(z)}{\eta(5z)}+7\eta^5(5z)\eta^6(z)+9\eta^{11}(5z)=
6q^\frac7{24}-65 q^\frac{31}{24}+291 q^\frac{55}{24}+\cdots\in S_\frac{11}2\(5,\ptfrac\bullet 5\nu^{7}\).
\end{equation}
We find that 
\begin{equation}\label{eq:F5mod13}
F_{13}\equiv \tilde F_{13}\pmod{13}.
\end{equation}
By  Theorem~\ref{thm:shimura-lift},  we have 
\begin{equation}\label{eq:liftf13}
\Sh_7(\tilde F_{13})=
6q-96q^2+486q^3+1536q^4+3376q^5+\cdots
\in S_{10}^{\new 2, 3}\(30, +1, -1\).
\end{equation}

There is a unique newform
\cite[\href{https://www.lmfdb.org/ModularForm/GL2/Q/holomorphic/30/10/a/c/}{30.10.a.c}]{lmfdb}
\begin{equation}
    g_{30} = q - 16q^2 + 81q^3 + 256q^4 - 625q^5+\cdots\in S_{10}^{\new}\(30, +1, -1\)
\end{equation}
and a unique newform \cite[\href{https://www.lmfdb.org/ModularForm/GL2/Q/holomorphic/6/10/a/a/}{6.10.a.a}]{lmfdb}
\begin{equation}\label{eq:gnewdef}
g_6=q - 16q^2 + 81q^3 + 256q^4 + 2694q^5+\cdots\in S_{10}^{\new}\(6, +1, -1\).
\end{equation}
We find by computing enough coefficients that
\begin{equation}\label{eq:f_13find}
\Sh_7(\tilde F_{13})=
\mfrac{2221}{537}\, g_6 - \mfrac{3544837}{537}\, g_6\|V_5 + \mfrac{1001}{537}\, g_{30}
\equiv 6g_6+4g_6\|V_5\pmod{13}.
\end{equation}

Suppose that $Q$ is a prime with 
\begin{equation}\label{eq:newform13}
T_Q g_6\equiv \beta_Q Q^4g_6\pmod{13}\ \ \ \ \text{with}\ \ \  \beta_Q\in \{\pm1\}.
\end{equation}
By \eqref{eq:f_13find} and the argument in Lemma~\ref{5.1 analogue}, we find that 
\begin{equation}\label{eq:heckefrobhalf}
T_{Q^2} \tilde F_{13}\equiv \pmfrac{12}Q\beta_QQ^4\,\tilde F_{13}\pmod{13}.
\end{equation}
Then Lemma~\ref{5.2 analogue} shows that 
\begin{equation}\label{eq:frob513cong}
 c\phi_{5}\pmfrac{13 Q^2 n+5}{24} \equiv 0  \pmod{13}\ \ \ \ \text{if}\ \ \ \pmfrac nQ=\beta_Q\pmfrac{-5}Q.
\end{equation}
By computing the eigenvalues of $g_6$, we find the following congruences for $Q<2000$ (this table can easily be expanded).
\setlength{\arrayrulewidth}{0.5mm}
\setlength{\tabcolsep}{18pt}
\renewcommand{\arraystretch}{2.2}

\begin{tabular}{ |p{1.9cm}|p{11cm}|  }
\hline
\multicolumn{2}{|c|}{$c\phi_5\pfrac{13 Q^2 n+5}{24}\equiv 0  \pmod{13}$ \ \  if \ \ \ $\pfrac nQ=\ep_Q$} \\
\hline
$\varepsilon_Q=1$& $Q=$ 103, 109, 283, 727, 769, 809, 991, 1063, 1223, 1231, 1259, 1291, 1307, 1367, 1409, 1543, 1733, 1789, 1831, 1861\\
\hline
$\varepsilon_Q=-1$& $Q=$ 97,
191, 241, 251, 397, 409, 439, 463, 751, 823, 839, 1229, 1277, 1321, 1361, 1621, 1657, 1933, 1979, 1993  \\
\hline
\end{tabular}

\vspace{1em}

As in \eqref{eq:congex}  each of these gives rise to many congruences of the form $c\phi_5(13 Q^3 n+\beta)\equiv 0\pmod{13}$ by selecting $n$ in residue classes modulo $24Q$.

\bibliographystyle{plain}
\bibliography{shimura-eta.bib}
\end{document}